\documentclass[11pt]{article}
\usepackage{amsmath,amsthm,amsfonts,amssymb,amscd, amsxtra, mathrsfs,textcomp,verbatim,color}
\pdfoutput=1
\usepackage[active]{srcltx}
\usepackage{url}
\usepackage{cite}
\usepackage[english]{babel}
\usepackage{graphicx}
\usepackage{subfigure}
\usepackage{booktabs}
\usepackage{multirow}
\usepackage[margin=1.0in]{geometry}
\newtheorem{theorem}{Theorem}
\newtheorem{lemma}{Lemma}
\newtheorem{corollary}{Corollary}
\newtheorem{proposition}{Proposition}
\newtheorem{remark}{Remark}
\newtheorem{definition}{Definition}

\newtheorem{algorithm}{Algorithm}

\DeclareMathOperator{\grad}{grad}

\newcommand{\m}{\mathbb{M}}
\newcommand{\tpm}{T_{p}\mathbb{M}}

\begin{document}

\title{On the globalization  of  Riemannian Newton method}

\author{
Bortoloti, M. A. A.  \thanks{Universidade Estadual do Sudoeste da Bahia, BA 45083-900, BR (e-mail:{\tt mbortoloti@uesb.edu.br}, {\tt telesfernandes@uesb.edu.br}). The author was supported in part by UESB.}
\and
Fernandes, T. A. \footnotemark[1]
\and
Ferreira, O. P.  \thanks{ Universidade Federal de Goias, Goiania, GO 74001-970, BR (e-mail:{\tt orizon@ufg.br}). The author was supported in part by CNPq Grants 305158/2014-7 and FAPEG/GO.} 
}
\maketitle
%%%%%%%%%%%%%%%%%%%%%%%%%%
\begin{abstract}
\noindent In the present paper, in order to find a singularity of a vector field defined on Riemannian manifolds, we present a new  globalization strategy  of Newton method and establish its global convergence with superlinear rate. In particular, this globalization generalizes for a general retraction the existing damped Newton's method. The  presented  global convergence analysis does not require any hypotesesis on singularity of the vector field. We applied the proposed method to solve the truncated singular value problem on the product of two Stiefel manifolds, the dextrous hand grasping problem on the cone of  symmetric  positive definite matrices and the Rayleigh quotient on the sphere. Moreover,   some academic  problems  are solved. Numerical experiments are presented showing that the proposed algorithm has better robustness compared  with  the aforementioned  method.

\end{abstract}
\vspace{0.7cm}

\noindent
{\bf Keywords:} Global Convergence $\cdot$ Riemannian Newton Method $\cdot$ Superlinear Rate $\cdot$ Retraction.
\vspace{0.5cm}

\noindent
 {\bf  2010 AMS Subject Classification:}  	90C30,  49M15, 65K05.

%%%%%%%%%%%%%%%%%%%%%%%%%%
\section{Introduction}\label{Introduction}
%%%%%%%%%%%%%%%%%%%%%%%%%%
Iterative methods on manifolds arise in the context of optimizing a real-valued function, dating back to the work of Luenberger \cite{Luenberger1972} in the early 1970s, if not earlier. Luenberger proposed the idea of performing a line search along geodesics that are computationally feasible. Around 1990, the main research issue was to exploit differential-geometric objects in order to formulate optimization strategies on abstract nonlinear manifolds. Gabay in \cite{Gabay1982} was the first to focus on optimization on manifolds by minimizing a differentiable function defined on a Riemannian manifold. In the 1990s, the field of optimization on manifolds gained considerable popularity, especially with the work of Edelman et al. \cite{EdelmanAriasSmith1999}. Recent years have witnessed a growing interest in the development of numerical algorithms for nonlinear manifolds, as there are many numerical problems posed in manifolds arising in various natural contexts. For example, eigenvalue problems  \cite{vandereycken2013,liu2015,wen2017, zhang2014, Chiang2019}, low-rank matrix completion \cite{wen2012}, loss minimization problem \cite{sato2019} and dextrous hand grasping problem \cite{robotic2002,robotic2007,robotic2002_2}. For such problems, the solutions of a system of equations often have to be computed or the zeros of a vector field have to be found. Because these problems are naturally posed on Riemannian manifolds, we can use the specific underlying geometric and algebraic structures to significantly reduce the computational cost of finding the zeros of a vector field. In this work, instead of focusing on finding singularities of gradient vector fields on Riemannian manifolds, which includes finding local minimizers, we consider the more general problem of finding singularities of vector fields. 

Newton's method is known to be a powerful tool for finding the zeros of nonlinear functions in Banach spaces. It also serves as a powerful theoretical tool with a wide range of applications in pure and applied mathematics \cite{Adler2002,Nash1956,Moser1961}. These factors have motivated several studies to investigate the issue of generalizing Newton's method from a linear setting to the Riemannian setting \cite{Absil2009,Li2009,Argyros2009,Schulz2014,Wang2009,Ferreira2012,Ferreira2002,Li2006,FernandesAndFerreiraAndYuan2017}. Although Newton's method shows fast local convergence, it is highly sensitive to the initial iterate and may diverge if the initial iterate is not sufficiently close to the solution. Thus, Newton's method does not converge in general. To overcome this drawback, some strategies have been introduced for using Newton's method in optimization problems, such as the Broyden--Fletcher--Goldfarb--Shanno (BFGS) algorithm, Levenberg--Marquardt algorithm, and trust region algorithm \cite{Dennis1996,Bertsekas2014}. When the objective function is twice continuously differentiable and strongly convex, the Newton direction is a descent direction of the objective function. Hence, by adjusting the step size in the Newton direction using, e.g., the Armijo rule, we can ensure convergence of Newton's method. This strategy of dumping the step size to globalize Newton's method is known as damped Newton's method. For a comprehensive study of this method, see \cite{Dennis1996,Solodov2014,Bertsekas2014,Burdakov1980}. For the problem of finding a zero of a nonlinear equation in a Euclidean setting, this strategy of dumping the Newton step size can also be adopted by using a merit function for which the Newton direction is a descent direction. This strategy was generalized from Euclidian context to Riemannian context, see \cite{Bortoloti2020}.

The generalization  to Riemannian setting of damped Newton's method by using a merit function  was obtained by using the notion of continuously moving  in the Newton direction while staying on a geodesic curve in the manifold until we reach a point where the vector field vanishes. By using the geodesic curve, we can  define the exponential mapping that can be used  to give a short notation  for a  geodesic with a given  starting  point and  initial  velocity. However, the geodesic, and consequently the exponential mapping, is defined as the solution of a nonlinear ordinary differential equation, whose efficient computation generally involves significant numerical challenges. Nevertheless, an approximation of the geodesic is sufficient to guarantee the desired convergence properties. Actually, to obtain the next iterate of an iterative method on a manifold, it is sufficient to use the notion of moving in the direction of a tangent vector while staying on the manifold. It is generalized by the notion of a retraction mapping that may generate a curve on the manifold with greater computational efficiency compared to the exponential mapping. The idea of using computationally efficient alternatives to the exponential mapping was introduced in \cite{Manton2002}. The strategy of using approximations of classical geometric concepts to obtain efficient iterative algorithms has attracted considerable attention lately in the context of Riemannian optimization; see, e.g., \cite{AbsilAndBakerAndGallivan2007,Absil2009,Ring2012,CambierAbsil2016, MR3093408, MR3093466}. Obtaining an iterative method using retraction is becoming increasingly common, as such algorithms are  faster and possibly more robust than existing algorithms. Recent studies on the development of geometric optimization algorithms that exploit the mapping retraction on nonlinear manifolds include \cite{HuangWenAbsilGallivan2015,HuangAndWenAndGallivan2017,HuangAndAbsilAndGallivan2017,Zhu2017,HuangAndGallivanAndAndAnujAndAbsil2016}. A  toolbox for building retractions on manifolds can be found in \cite{AbsilandMalick2012}.

Our main contribution is to present a Newton type algorithm with global convergence. From the theoretical point  of view this algorithm generates  a sequence that converges  without any assumption on singularity  of the vector field in consideration, which improves the convergence analysis of \cite{Aihara2017, Bortoloti2020}.   Besides,  this algorithm uses general retractions instead of only the exponential mapping. In addition,  numerical experiments are  presented  showing  that  the proposed  algorithm has better  robustness compared with  the one presented in  \cite{Aihara2017, Bortoloti2020} and better  performance than the one presented in \cite{Bortoloti2020}.   In order to present a numerical performance for proposed algorithm  we have submitted it to the task to solve the truncated singular value problem on the product of two Stiefel manifolds \cite{Sato2013}, the Rayleihg quotient on the sphere and the dextrous hand grasping problem, see \cite{robotic2002,robotic2007,robotic2002_2}. Also, we yet analyse the problem to find the singularity of a nonconservative vector field on the sphere and to minimize an academic function on the cone of symmetric definite positive matrices.

 The remainder of this paper is organized into five sections. The basic concepts  and auxiliary results are developed in Section \ref{sec:basic}. In Section \ref{sec:supelinear}, we present a local superlinear convergence of the Newton method with retraction. The main result  of the paper is presented in Section \ref{sec:DampedNewtonMethod}. Concrete examples and numerical experiments of the new gained insights of the proposed method are presented in Section \ref{Sec:Numerical_Experiment}. Concluding remarks are presented in Section \ref{sec:conclusions}.

%%%%%%%%%%%%%%%%%%%%%%%%%%%
\section{Preliminaries}\label{sec:basic}
%%%%%%%%%%%%%%%%%%%%%%%%%%%
In this section, we recall some notations, definitions, and basic properties of Riemannian manifolds used throughout the paper, which can be found in many introductory books on Riemannian geometry, for example \cite{doCarmo1992} and \cite{Sakai1996}. 

For a smooth manifold $\mathbb{M}$, denote the {\it tangent space} of $\m$ at $p$ by $\tpm$ and the {\it tangent bundle} of $\m$ by $T\m=\bigcup_{p\in\m}\tpm$.  The corresponding norm associated with the Riemannian metric $\langle \cdot ~, ~ \cdot \rangle$ is denoted by $\| \cdot \|$. The Riemannian  distance  between $p$ and $q$   in a finite-dimensional Riemannian manifold $\mathbb{M}$ is denoted  by $d(p,q)$, and it induces the original topology on $\mathbb{M}$. An open ball of radius $r>0$ centered at $p$ is defined as  $B_{r}(p):=\left\lbrace q\in\m:d(p,q)<r\right\rbrace$.  Let  $\Omega \subset  \m$ be an open set, and let ${\cal X}(\Omega)$ denote the space of  $C^1$ vector fields on $\Omega$.  Let $\nabla$ be the Levi-Civita connection associated with $(\mathbb{M}, \langle \cdot ~, ~ \cdot \rangle)$.  The covariant derivative of $X \in {\cal X}(\Omega)$ denoted by $\nabla$ defines at each $p\in \Omega$ a linear map $\nabla X(p):\tpm\to\tpm$ given by $\nabla X(p)v:=\nabla_{Y}X(p)$, where $Y$ is a vector field such that $Y(p)=v$.   For $f: \mathbb{M} \to \mathbb{R}$,  a twice-differentiable function the Riemannian metric induces the mappings   $f\mapsto  \mbox{grad} f $ and   $f\mapsto \mbox{Hess} f$, which associate    its {\it gradient} and {\it Hessian} via the rules $d f(X):=\langle \mbox{grad} f,X\rangle$  and  $d^2 f(X, X):=\mbox{Hess} f X,X\rangle$, for all $X \in {\cal X}(\Omega)$, respectively.  Therefore,   $\mbox{Hess} f X= \nabla_{X}  \mbox{grad} f$, for all $X \in {\cal X}(\Omega)$. The {\it norm of a  linear map} $A:T_p \mathbb{M} \to T_p \mathbb{M}$  is defined by $\|A\|:=\sup \left\{ \|Av \|~:~ v\in T_p \mathbb{M}, \,\|v\|=1 \right\}$.  A vector field $V$ along a differentiable curve $\gamma$ in $\mathbb{M}$ is said to be {\it parallel} iff $\nabla_{\gamma^{\prime}} V=0$.  For each $t \in [a,b]$, the operator $\nabla$ induces an isometry relative to $ \langle \cdot , \cdot \rangle  $, $P_{\gamma,a,t} \colon T _{\gamma(a)} {\mathbb{M}} \to T _{\gamma(t)} {\mathbb{M}}$, defined by $ P_{\gamma,a,t}\, v = V(t)$, where $V$ is the unique vector field on $\gamma$ such that
$\nabla_{\gamma'(t)}V(t) = 0$ and $V(a)=v$, the so-called {\it parallel transport} along  of a segment  of curve  $\gamma$ joining the points   $\gamma(a)$  and  $\gamma(t)$. Further, note that $ P_{\gamma,\,b_{1},\,b_{2}}\circ P_{\gamma,\,a,\,b_{1}}=P_{\gamma,\,a,\,b_{2}}$ and  $P_{\gamma,\,b,\,a}=P^{-1}_{\gamma,\,a,\,b}$. As long as there is no confusion, we will  consider the notation $P_{pq}$  instead of $P_{\gamma,\,a,\,b}$ when $\gamma$ is the unique segment of curve joining  $p$ and $q$. The following lemma   ensures that, if $\nabla X(\bar{p})$ is nonsingular then there exists a neighborhood of $\bar{p}$ such that $\nabla X$ is also nonsingular.
\begin{lemma}\label{le:NonSing}
 Assume that $\nabla X$ is  continuous at ${\bar p}$. Then, ${\lim_{p\to {\bar p}}}\left \| P_{p{\bar p}}\nabla X(p)P_{{\bar p}p}-\nabla X({\bar p}) \right \|=0.$ Moreover, if  $\nabla X({\bar p})$ is  nonsingular,  then there exists $0<\bar{\delta}<\delta_{{\bar p}}$ such that $B_{\bar{\delta}}({\bar p})\subset \Omega$, and  for each $p\in B_{\bar{\delta}}({\bar p})$,   $\nabla X(p)$ is nonsingular and $\left \| \nabla X(p)^{-1}\right \|\leq 2\left \|\nabla X({\bar p})^{-1}\right \|$.
\end{lemma}
\begin{proof}
See \cite[Lemma 3.2]{FernandesAndFerreiraAndYuan2017}
\end{proof}
In the following we present the concept of  retraction which has been introduced by \cite{Manton2002}.  
\begin{definition}\label{Df:nqc}
A \textit{retraction} on a manifold $\mathbb{M}$ is a smooth mapping $R$ from the tangent bundle $T\mathbb{M}$ onto $\mathbb{M}$ with the following properties: If $R_{p}$ denote the restriction of $R$ to $T_{p}\mathbb{M}$, then
\begin{itemize}
\item[(i)] $R_{p}(0_{p})=p$, where $0_{p}$ denotes the origin of $T_{p}\mathbb{M}$;
\item[(ii)] With the canonical identification $T_{0_{p}}T_{p}\mathbb{M}\simeq T_{p}\mathbb{M}$, $R'_{p}(0_{p})=I_{p}$, where $I_{p}$ is the identity mapping on $T_{p}\mathbb{M}$, and $R'_{p}$ denotes the diferential of $R_{p}$.
\end{itemize}
\end{definition}
The Definition~\ref{Df:nqc} implies that the  \textit{exponential map}  is a retraction, see \cite{Absil2009}.  Since $R'_{p}(0_{p})=I_{p}$, by the Inverse Function Theorem, $R_{p}$ is a local diffeomorphism. Hence, we  define the {\it injectivity radius} of $\mathbb{M}$ at $p$ with respect to $R$ as follow $i_{p}:=\sup\left\{ r>0~:~{R_{p}}_{\lvert_{B_{r}(0_{p})}} \mbox{ is\, a\, diffeomorphism} \right\},$
where $B_r(0_{p}):=\left\lbrace  v\in T_{p}\mathbb{M}:\parallel v-0_{p}\parallel <r\right\rbrace$.
\begin{remark}\label{unicidadedageodesica}
Let $\bar{p}\in \mathbb{M}$. The above definition implies  that if $0<\delta<i_{\bar{p}}$, then $R_{\bar p}B_{\delta}(0_{ \bar p})=B_{\delta}( \bar p)$. Moreover, for all $p\in B_{\delta}(\bar p)$  the  curve segment $\gamma_{\bar{p}p}(t)=R_{\bar{p}}\left(t R^{-1}_{\bar p}p\right)$ joining  $\bar p$ to $p$ belongs to $B_{\delta}( \bar p)$.
\end{remark}
The following result  establishes an important  relation between the retraction and the Riemannian distance and its proof can be found in \cite[Lemma 6]{Ring2012}.
\begin{proposition}\label{lm:kwqdklu}
Let $\mathbb{M}$ be a Riemannian  manifold endowed with a retraction $R$ having equicontinuous derivatives in a neighborhood  of  $\bar{p}\in\mathbb{M}$. Then, there exist $a_{0}>0$, $a_{1}>0$, and $\delta_{a_{0},a_{1}}$ such that for all $p$ in a sufficiently small neighborhood of $\bar{p}$ and all $v\in T_{p}\mathbb{M}$ with $\|v\|\leq\delta_{a_{0},a_{1}}$, the following inequality holds
\begin{equation}\label{eq:bdr}
a_{0}\|v\|\leq d\left(p,~R_{p}(v)\right)\leq a_{1}\|v\|.
\end{equation}
\end{proposition}
Let $a_{0}>0$, $a_{1}>0$, and $\delta_{a_{0},a_{1}}$ be given as in the Lemma~\ref{lm:kwqdklu}. Making  $\delta_{a_{0},a_{1}}$ smaller, if necessary,  such that $\delta_{a_{0},a_{1}}<i_{\bar{p}}$. Let $v\in B_{\delta_{a_{0},a_{1}}}(0_{\bar{p}})$ and $p=R_{\bar{p}}(v)$. Then,  by  \eqref{eq:bdr} we can conclude
\begin{equation}\label{eq:qwhj1}
a_{0}\left\|R^{-1}_{\bar{p}}(p)\right\|\leq d\left(\bar{p}, ~p\right)\leq a_{1}\left\|R^{-1}_{\bar{p}}(p)\right\|,\qquad \forall ~ p\in B_{\delta}(\bar{p}), 
\end{equation}
where $\delta <\delta_{a_{0},a_{1}}$.
Now, we are ready to define the number $K_{R,p}$. Letting  $i_{p}$ be the radius of injectivity of $\mathbb{M}$ at $p$ we define the quantity 
\begin{equation}\label{eq:vizinjetiv}\nonumber
\delta_{p}:=\min\{1, i_{p}\}.
\end{equation} 
\begin{equation} \label{eq:Kp}
K_{R,p}:=\sup\left \{  \dfrac{d(R_{q}u, R_{q}v)}{\parallel u-v\parallel} ~:~ q\in B_{\delta_{p}}(p), ~ u,\,v\in T_{q}\mathbb{M}, ~u\neq v, ~\parallel v\parallel\leq \delta_{p},~\parallel u-v\parallel\leq \delta_{p}\right\}.
\end{equation}

Let  $i_{p}$ be the radius of injectivity of $\mathbb{M}$ at $p$  and define the quantity $\delta_{p}:=\min\{1, i_{p}\}$. Consider  $X\in {\cal X}(\Omega)$   and $  \bar{p}\in \Omega$. Assume that  $0<\delta<\delta_{\bar{p}}$. From definition of $\delta_{p}$ follows that for any curve $[0,~1]\owns t\mapsto\gamma(t)=R_{\bar{p}}\left(tv_{\bar{p}}\right)$  joining  $\bar{p}$ to $p\in B_{\delta}(\bar{p})$ such that $\gamma'(0)=v_{\bar{p}}$, we have $v_{\bar{p}}=R^{-1}_{\bar{p}}p$. Moreover, using \cite[ equality 2.3]{Ferreira2002} we obtain
\begin{equation}\label{diferenciabilidade}
X(p)= P_{\bar{p}p}X(\bar{p})+P_{\bar{p}p}\nabla X(\bar{p})R^{-1}_{\bar{p}}p+ \|R^{-1}_{\bar{p}}p\|r(p),  \qquad \underset{p\to \bar{p}}{\lim}r(p)=0, 
\end{equation}
for each $p\in   B_{\delta}(\bar{p})$.

We end  this section by formally presenting the problem of interest in this paper. Let  $X:\mathbb{M}\to T\mathbb{M}$ with $X(p)\in T_p\mathbb{M}$ be a differentiable vector field. We are interested in to find a  $p\in \mathbb{M}$ such that
\begin{equation} \label{eq:TheProblem0}
 X(p)=0. 
\end{equation}
%%%%%%%%%%%%%%%%%%%%%%%%%%
\section{Local superlinear convergence of  Newton method}  \label{sec:supelinear}
%%%%%%%%%%%%%%%%%%%%%%%%%%
In this section, we analyse the  local convergence of Newton method with a general retraction to solve the problem \eqref{eq:TheProblem0}, which generalize the  results   presented in  \cite{FernandesAndFerreiraAndYuan2017}. We first formally present Newton method with a general retraction. It is described as follow.

\noindent
\\
\hrule
\begin{algorithm} \label{AL:NewtonPuro}
{\bf Newton Method\\}
\hrule
\begin{description}
\item[Step 0.] Take an initial point $p_0 \in \mathbb{M}$, and set $k=0$.
\item[Step 1.] Compute {\it search direction} $v_{k}\in T_{p_{k}}\mathbb{M}$ as a solution of the linear equation
\begin{equation}\label{eq:EQNEWTONPURO12}
X(p_{k})+\nabla X(p_{k})v=0.
\end{equation}
\hspace{.6cm }If $v_{k}$ exists go to {\bf Step $2$}. Otherwise, stop.
\item[Step 2.] Compute 
\begin{equation} \label{eq:NM12PURO}
p_{k+1}:=R_{p_{k}}v_{k}.
\end{equation}
\item[Step 3.] Set $k\leftarrow k+1$ and go to {\bf Step 1}
\end{description}
\hrule
\end{algorithm}
\vspace{0.3cm}
When the retraction $R$ is the exponential mapping, a sequence generated by this method converges to a singularity of $X$ with superlinear rate,  \cite{FernandesAndFerreiraAndYuan2017}. Moreover, if  the covariant derivative  is Lipschitz continuous around the singularity then the method has $Q$-quadratic convergence rate, see \cite{Ferreira2012}. It is well known that the convergence of a sequence generate by Algorithm~\ref{AL:NewtonPuro} is ensured when the initial guess is sufficiently  close to a solution at which the covariant  derivative is nonsingular. Otherwise, the equation \eqref{eq:EQNEWTONPURO12} may not have a solution. In this case, the  Algorithm~\ref{AL:NewtonPuro} stops. In the next section, we present a new algorithm that overcome this. Now, our aim is to prove a generalization of \cite[Theorem 3.1]{FernandesAndFerreiraAndYuan2017}, it is, under  the assumption  of nonsingularity of the covariant derivative at the solution ${\bar p}$, the  iteration  \eqref{eq:NM12PURO} starting in a suitable neighborhood of  ${\bar p}$  is well defined and converges superlinearly to ${\bar p}$. Before to obtain that generalization some results are required.

Consider $\bar{\delta}>0$  given by Lemma~\ref{le:NonSing} and define {\it Newton's iterate mapping}   $N_{R,X}:B_{\bar{\delta}}(\bar{p}) \to  \mathbb{M} \nonumber $  by
\begin{equation} \label{Newton's iterate mapping}
N_{R,X}(p):=R_{p}(-\nabla X(p)^{-1}X(p)).
\end{equation}
 The next lemma ensures  existence  of a neighborhood of $\bar{p}$ where Newton's iterate given by \eqref{eq:NM12PURO} belong to the same neighborhood. 
\begin{lemma}\label{L5} 
Let  ${\bar p}\in\mathbb{M}$ a solution of \eqref{eq:TheProblem0}. Assume  that $\nabla X$ is continuous at ${\bar p}$ and $\nabla X({\bar p})$ is  nonsingular.  Then, 
$$
\lim_{p\to {\bar p}}\frac{d(N_{R,X}(p),{\bar p})}{d(p,{\bar p})}=0.
$$
\end{lemma}
\begin{proof}

 Define $r(p):= X(p)- P_{{\bar p}p}X({\bar p})-P_{{\bar p}p}\nabla X({\bar p}){R^{-1}_{\bar p}p}$, for each $p\in B_{\bar{\delta}}({\bar p})$,  where $\bar{\delta} $ is given by Lemma \ref{le:NonSing}. By using some algebraic manipulations  we obtain the following equality
$$
 \nabla X(p)^{-1}X(p)+R^{-1}_p{\bar p}=\nabla X(p)^{-1} \big{[ }r(p)+ \left[ P_{{\bar p}p}\nabla X({\bar p})-\nabla X(p)P_{{\bar p}p}\right]R^{-1}_{\bar p}p\big{]}.
$$
Thus, using the above equality, the definition of $r$, and some  properties of the norm,  we  conclude
$$
\dfrac{\left\|\nabla X(p)^{-1}X(p)+R^{-1}_p{\bar p} \right\|}{d(p,~{\bar p})}\leq\left\|\nabla X(p)^{-1}\right \| \big{[} \left\| r(p) \right\|+
\left\| P_{{\bar p}p}\nabla X({\bar p})-\nabla X(p)P_{{\bar p}p}\right\|\big{]}\dfrac{ \left\|R^{-1}_{\bar p}p\right\|}{d(p,~{\bar p})}.
$$
By using \eqref{eq:qwhj1} we obtain $\left\|R^{-1}_{\bar p}p\right\|/d(p,~{\bar p})\leq1/a_{0}$. Combining these two last inequalities we have
$$
\dfrac{\left\|\nabla X(p)^{-1}X(p)+R^{-1}_p{\bar p} \right\|}{d(p,~{\bar p})}\leq\dfrac{\left\|\nabla X(p)^{-1}\right \|}{a_{0}} \left[ \left\| r(p) \right\|+
\left\| P_{{\bar p}p}\nabla X({\bar p})-\nabla X(p)P_{{\bar p}p}\right\|\right].
$$
Since $p\in  B_{\bar{\delta}}({{\bar p}})$, $P_{{\bar p}p}P_{p{\bar p}}=I_{p}$ where $I_{p}$ denotes the identity operator on $T_{p}\mathbb{M}$, and the parallel transport is an isometry, Lemma \ref{le:NonSing} combined with this last inequality implies 
\begin{equation}\label{eq:hel}
\left\|\nabla X(p)^{-1}X(p)+R^{-1}_p{\bar p} \right\|\leq
\dfrac{2}{a_{0}}\left\|\nabla X({\bar p})^{-1}\right\|  \left[ \left\| r(p) \right\|+ \left\| P_{p{\bar p}}\nabla X(p)P_{{\bar p}p}-\nabla X({\bar p})\right\| \right]d(p,~{\bar p}).
\end{equation}
Owing to Lemma~\ref{le:NonSing} and  $ {\lim_{p\to {\bar p}}}r(p)=0$,  the right-hand side of the last inequality tends to zero, as $p$ goes to ${\bar p}$. Recalling that  $\delta_{{\bar p}}=\min\{1, i_{{\bar p}}\}$, we  can shrink $ \bar\delta$,  if necessary,  to obtain
$$
\left\|  \nabla X(p)^{-1}X(p)+R^{-1}_p{\bar p} \right\|\leq \delta_{{\bar p}},\qquad  \forall ~p\in   B_{ \bar\delta}({\bar p}).
$$
Hence, from definitions of $N_{R,X}$ in \eqref{Newton's iterate mapping} and  $K_{R,{\bar p}} $ in \eqref{eq:Kp}, we can conclude
$$
d(N_{R,X}(p),{\bar p})\leq  K_{{\bar p}}\left\| - \nabla X(p)^{-1}X(p)-R^{-1}_p{\bar p} \right\|,\qquad  \forall ~p\in   B_{ \bar\delta}({\bar p}).
$$
Therefore, by combining \eqref{eq:hel} with  the last  inequality  we obtain  for all $p\in B_{ \bar\delta}({\bar p})$ that
$$
\dfrac{d(N_{R,X}(p),{\bar p})}{d(p, {\bar p})}\leq \dfrac{2}{a_{0}} K_{{\bar p}} \left\|\nabla X({\bar p})^{-1}\right\|  \left[ \left\| r(p) \right\|+ \left\| P_{p{\bar p}}\nabla X(p)P_{{\bar p}p}-\nabla X({\bar p})\right\| \right].
$$
By letting $p$  tend to ${\bar p}$ in the last inequality, by considering Lemma~\ref{le:NonSing} and that  $r(p)$ tends to zero, as $p$ goes to ${\bar p}$,  the desired result follows.
\end{proof}
In the next we show that whenever the vector field  is continuous and has  nonsingular  covariant derivative at a solution, there exist a neighborhood  around of it which  is  invariant by the Newton's iterate mapping associated. Its proof is a direct consequence from Lemma~\ref{le:NonSing} and Lemma~\ref{L5}.
\begin{lemma}\label{L:InvNx}
Let  $\bar{p}\in\mathbb{M}$ such that $X(\bar p)=0$. If $\nabla X$ is continuous at $\bar p$ and $\nabla X(\bar{p})$ is nonsingular,  then  there exists $0<\hat{\delta}<\delta_{\bar{p}}$ such that $B_{\hat{\delta}}(\bar{p})\subset \Omega$ and $\nabla X(p)$ is nonsingular  for each $p\in B_{\hat{\delta}}(\bar{p})$. Moreover,    $N_{R,X}(p)\in B_{\hat{\delta}}(\bar{p})$, for all $p\in B_{\hat{\delta}}(\bar{p})$.
\end{lemma}
The main result this section is presented as follow. It is a generalization of \cite[Theorem 3.1]{FernandesAndFerreiraAndYuan2017} for a general retraction. Its proof can be made by adaptations of the ideia used in that result. For this reason, we do not present it here.
\begin{theorem}\label{Theorem of Newton}
Let $\mathbb{M}$ be a Riemannian manifold with a retraction $R$ and $\Omega\subset\mathbb{M}$ be an open set. Let $X:\Omega\to T\mathbb{M}$ be a differentiable vector field and ${\bar p}\in\Omega$. Consider the Newton sequence $\{p_k\}$ generated by Algorithm~\ref{AL:NewtonPuro}. Suppose that ${\bar p}$ is a singularity of $X$ and  $\nabla X$ is continuous  and   nonsingular at ${\bar p}$. Then,  there exists  $\bar{\delta}>0$ such that,  for  all  $p_{0}\in B_{\bar{\delta}}({\bar p})$, the sequence $\{p_k\}$ is well defined, contained in $B_{\bar{\delta}}({\bar p})$, and it converges superlinearly to ${\bar p}$.
\end{theorem}
Let $X=\ \mbox{grad}\,f$, then the following result is a version of Theorem~\ref{Theorem of Newton} for finding  critical points of a twice-differentiable function.
\begin{corollary}\label{corolriodopricipal}
Let $\mathbb{M}$ be a Riemannian manifold with a retraction $R$ and $\Omega\subset\mathbb{M}$ be an open set. Let $f:\Omega\to \mathbb{R}$ be a twice-differentiable function and ${\bar p}\in\Omega$. Suppose that ${\bar p}$ is a critical point of $f$ and  $\emph{Hess}\,f$ is continuous and nonsingular at ${\bar p}$.  Then,  there exists $\bar{\delta}>0$ such that,  for  all  $p_{0}\in B_{\bar{\delta}}({\bar p})$, a sequence generated by Algorithm~\ref{AL:NewtonPuro},
\begin{equation} \label{eq:qwkejf}
p_{k+1}=R_{p_{k}}(-\emph{Hess}\,f(p_{k})^{-1}\emph{grad}\,f(p_{k})),  \qquad k=0, 1, \ldots.
\end{equation}
is well defined, contained in $B_{\bar{\delta}}({\bar p})$ and converges superlinearly to ${\bar p}$.
\end{corollary}
%%%%%%%%%%%%%%%%%%%%%%%%%%%
\section{Globalization of  Newton method}\label{sec:DampedNewtonMethod}
%%%%%%%%%%%%%%%%%%%%%%%%%%%
 In \cite{Bortoloti2020} has been presented a global version of the Newton method with the iterations being updated by the exponential mapping. In this section, we present a version of that method for a general retraction, see Algorithm~\ref{AL:jogo_with_retraction} below.  Our  numerical experiments in the present paper have shown that this method is quite sensitive with respect to retractions used,  having  effect on its robustness. To correct this drawback, we will also  present a new version of this method, see  Algorithm~\ref{AL:dampedRetraction2}. Besides, its  convergence analysis is carried out under weaker conditions. In order to present both  algorithms, we consider a {\it merit function}   $\varphi:\mathbb{M}\to\mathbb{R}$ associated to the vector field $X$, which is defined as
\begin{equation}\label{eq:naturalfuncaomerito}
\varphi(p)=\dfrac{1}{2}\parallel X(p)\parallel^{2}.
\end{equation}

%%%%%%%%%%%%%%%%%%%%%%%%%%%
\subsection{Damped Newton method}\label{sec:DampedNewtonMethod1}
%%%%%%%%%%%%%%%%%%%%%%%%%%%
In the following we present a version of the algorithm introduced in \cite{Bortoloti2020}  for a general retraction. This algorithm is similar to the one presented in \cite{Bortoloti2020}, for the sake of completeness and support in the next section, we have included it here.The formal statement  of the algorithm is as follows. 
\noindent
\\
\hrule
\begin{algorithm} \label{AL:jogo_with_retraction}
{\bf Damped Newton method\\}
\hrule
\begin{description}
\item[Step 0.] Choose a scalar $\sigma\in(0,1/2)$, take an initial point $p_0 \in \mathbb{M}$, and set $k=0$;
\item[Step 1.] Compute {\it search direction} $v_{k}\in T_{p_{k}}\mathbb{M}$ as a solution of the linear equation
\begin{equation}\label{eq:Newton1}
X(p_{k})+\nabla X(p_{k})v=0.
\end{equation}
\hspace{.6cm }If $v_{k}$ exists go to {\bf Step $2$}. Otherwise, set the search direction as  $v_{k}=  -\grad\,\varphi(p_{k})$,  where $\varphi$ is defined by \eqref{eq:naturalfuncaomerito}, i.e.,  
\begin{equation}\label{eq:DirectionOfGradiente}
v_{k}=-\nabla X(p_{k})^{*}X(p_{k}).
\end{equation}
\hspace{.5cm} If $v_{k}=0$, stop.
\item[Step 2.] Compute the {\it stepsize} by the rule 
\begin{equation}\label{eq:armijo1}
    \alpha_k:=\max \left\{ 2^{-j}~: ~\varphi\left(R_{p_k}(2^{-j} v_{k})\right)\leq \varphi(p_k)+\sigma 2^{-j}\left\langle \grad\,\varphi(p_k), ~ v_{k} \right\rangle, ~j\in \mathbb{N} \right\};
\end{equation}
\hspace{.6cm} and set the {\it next iterated}  as 
\begin{equation} \label{eq:NM12331}
p_{k+1}:=R_{p_{k}}(\alpha_{k}v_{k});
\end{equation}
\item[Step 3.] Set $k\leftarrow k+1$ and go to {\bf Step 1}.
\end{description}
\hrule
\end{algorithm}
\vspace{0.3cm}

We can see that Algorithm~\ref{AL:jogo_with_retraction} is a generalized version of the algorithm considered in \cite{Bortoloti2020}, by using a general retraction in {\bf Step 2}.   To analyze a sequence generated by the method studied in \cite{Bortoloti2020} it   was necessary  to assume  nonsingularty  of covariant derivative of the vector field  at its cluster points. This assumption is also required  here,  it allows us  to obtain the same result of the ones obtained in \cite{Bortoloti2020}.  Next we state the convergence theorem  to sequence generated by Algorithm~\ref{AL:jogo_with_retraction}.  Since its proof can be made by adaptations of the ideia used in   \cite{Bortoloti2020} to general retraction, we do not present it here.

\begin{theorem}\label{Th:dampedRetraction1}
Let $\mathbb{M}$ be a Riemannian manifold,  $\Omega\subset\mathbb{M}$ be an open set and  $X:\Omega\to T\mathbb{M}$ be a differentiable vector field.  Take  $R$ a retraction in $\mathbb{M}$.  Assume $\left\{p_{k}\right\}$, generated by  Algorithm~\ref{AL:jogo_with_retraction}, has  an accumulation point   $\bar p\in \Omega$ and   $\nabla X$ is  continuous and  nonsingular at $\bar p$. Then, $\left\{p_{k}\right\}$  converges superlinearly to $\bar p$ and is a  singularity of $X$.  
\end{theorem}

It is well known that the superliner rate of Newton method just can be reached when the covariant derivative is nonsingular at the solution. However, this assumption is not necessary to obtain convergence of the damped Newton method.  On the other hand, the  equation  \eqref{eq:Newton1} at points far away a singularity may has more than one solution. In that case, a sufficient decreasing  of the merit function is not guaranteed, which implies  a great computational effort  of  linear search in {\bf Step~2}.  Consequently,   the robustness of the method is affected. In the next section, we will present a condition that excludes those solutions of  \eqref{eq:Newton1} that do not ensure sufficient decreasing  of the merit function.

%%%%%%%%%%%%%%%%%%%%%%%%%%%
\subsection{Modified damped Newton method}\label{sec:DampedNewtonMethod2}
%%%%%%%%%%%%%%%%%%%%%%%%%%%
In this section, we state the main algorithm of present paper and its  global convergence analysis.  This algorithm has an extra condition on the Newtonian direction in order to improve the  theoretical results  and numerical performance of Algorithm~\ref{AL:jogo_with_retraction}. The statement of the algorithm is as follows.

\noindent
\\
\hrule
\begin{algorithm} \label{AL:dampedRetraction2}
{\bf Modified damped Newton method\\}

\hrule
\begin{description}
\item[Step 0.] Choose a scalar $\sigma\in(0,1/2)$, $\theta\in[0,1]$, take an initial point $p_0 \in \mathbb{M}$, and set $k=0$;
\item[Step 1.] Compute {\it search direction} $v_{k}\in T_{p_{k}}\mathbb{M}$ as a solution of the linear equation
\begin{equation}\label{eq:Newton}
X(p_{k})+\nabla X(p_{k})v=0.
\end{equation}
\hspace{.6cm }If $v_{k}$ exists and

\begin{equation}\label{eq:NewtonDirectionCondction}
  \langle \grad\,\varphi(p_{k}), v_{k} \rangle \leq -{\theta}\|\grad\, \varphi(p_{k})\| \| v_{k} \|
\end{equation}
go to {\bf Step $2$}. Otherwise, set the search direction as  $v_{k}=  -\grad\,\varphi(p_{k})$,  where $\varphi$ is defined by \eqref{eq:naturalfuncaomerito}, i.e.,  
\begin{equation}\label{eq:DirectionOfGradiente}
v_{k}=-\nabla X(p_{k})^{*}X(p_{k}).
\end{equation}
\hspace{.5cm} If $v_{k}=0$, stop.
\item[Step 2.] Compute the {\it stepsize} by the rule 
\begin{equation}\label{eq:armijo}
    \alpha_k:=\max \left\{ 2^{-j}~: ~\varphi\left(R_{p_k}(2^{-j} v_{k})\right)\leq \varphi(p_k)+\sigma 2^{-j}\left\langle \grad\,\varphi(p_k), ~ v_{k} \right\rangle, ~j\in \mathbb{N} \right\};
\end{equation}
\hspace{.6cm} and set the {\it next iterated}  as 
\begin{equation} \label{eq:NM1233}
p_{k+1}:=R_{p_{k}}(\alpha_{k}v_{k});
\end{equation}
\item[Step 3.] Set $k\leftarrow k+1$ and go to {\bf Step 1}.
\end{description}
\hrule
\end{algorithm}
\vspace{0.3cm}

 Let us  describe the main features of Algorithm~\ref{AL:dampedRetraction2}.   We first compute $v_k$ a solution of  \eqref{eq:Newton}  if any,  and then we  check if it satisfies \eqref{eq:NewtonDirectionCondction}. In this case, we use it as a search direction in {\bf Step 2}. On the other hand, if $v_k$ does not satisfy either \eqref{eq:Newton} or \eqref{eq:NewtonDirectionCondction}, then we set the steepest descent direction $v_{k}=  -\grad\,\varphi(p_{k})$ as the search direction in {\bf Step~2}. In fact,     is the steepest descent  direction for the merit function \eqref{eq:naturalfuncaomerito}. Finally, we use the Armijo's  linear search  \eqref{eq:armijo} to compute a step-size $\alpha_k$. Then,  for a given retraction fixed previously,  from the current $p_k$  we compute the next iterated $p_{k+1}$   by   \eqref{eq:NM1233}. 
 \begin{remark} 
We point out for $\theta=0$,  Algorithm~\ref{AL:dampedRetraction2}  with the retraction being the exponential map  retrieve the algorithm considered in \cite{Bortoloti2020}. Indeed, if  $v_k\neq 0$ satisfies  \eqref{eq:Newton}, then  we have   
$$
\langle \grad\,\varphi(p_{k}), v_{k} \rangle=-\|X(p_k)\|^2 <0,
$$
which  implies  that  the condition  \eqref{eq:NewtonDirectionCondction}  holds trivially  for $\theta=0$.   Also, note that the condition $\theta=1$ it is the most restrictive among all those in the range $[0, 1]$. Because,  the condition \eqref{eq:NewtonDirectionCondction} happens only when $\grad\,\varphi(p_{k})$ and $v_{k}$ are collinear. 
\end{remark}
Before studying  the properties of  the sequence generated by the Algorithm \ref{AL:dampedRetraction2} it is need  some preliminaries results. 
 We begin  with a  useful result for  establishing   the well-definition of this sequence, which the proof can be found in \cite[Lemma 3]{Bortoloti2020}. 
 \begin{lemma}\label{L:DescDir}
Let $p\in \Omega$ such that  $X(p)\neq 0$.  Assume that   $v=-\nabla X(p)^{*}X(p)$ or  that  $v$   is a solution of the  linear equation 
$X(p)+\nabla X(p)v=0.$ If $v\neq 0$, then $ \left\langle\grad\,\varphi(p),  v\right\rangle <0$ .
\end{lemma}
Under suitable assumptions  the  following result  guarantees that the conditions \eqref{eq:Newton} and \eqref{eq:NewtonDirectionCondction} are  satisfied  in a neighbourhood  of a point where the covariant derivative is nonsingular. 
\begin{lemma}\label{Le:AdTesteNew}
 If $\nabla X$ is continuous at $\bar p\in \Omega$ and $\nabla X(\bar p)$ is nonsingular,  then  there exists $0<\hat{\delta}<\delta_{\bar{p}}$ such that $B_{\hat{\delta}}(\bar{p})\subset \Omega$, $\nabla X(p)$ is nonsingular  for  $p\in B_{\hat{\delta}}(\bar{p})$.  Moreover, for all  $p\in B_{\hat{\delta}}(\bar{p})$  the vector $v=-\nabla X(p)^{-1}X(p)$ is the  unique solution of the linear equation  $X(p)+\nabla X(p)v=0$   and,   for  $0\leq \theta < 1/{\mbox{cond}\, (\nabla X({\bar p}))}$,  there holds
\begin{equation}\label{eq:AdTesteNew}
\langle \grad\,\varphi(p), v \rangle \leq -{\theta}\|\grad\, \varphi(p)\| \| v\|, \qquad \forall ~p\in B_{\hat{\delta}}(\bar{p}).
\end{equation}
\end{lemma}
\begin{proof}
The first part of the proof follows from Lemma~\ref{le:NonSing}.  Whenever   $X(p)=0$,  the inequality \eqref{eq:AdTesteNew} holds. Assume that  $X(p)\neq 0$, for all $p\in B_{\hat{\delta}}(\bar{p})$.  Letting  $p\in B_{{\bar \delta}}(\bar{p})$,  we obtain  that   $\nabla X(p)$ is nonsingular. Thus,  $v=-\nabla X(p)^{-1}X(p)$ is the  unique solution of   $X(p)+\nabla X(p)v=0$, and due to  $\nabla \varphi(p)=\nabla X(p)^{*}X(p)$  and $\|\nabla X(p)^{*}\|=\|\nabla X(p)\|$ we conclude that 
$$
  \frac{\langle \grad\,\varphi(p), -v \rangle}{\|\grad\, \varphi(p)\| \| v\| }  = \frac{\|X(p)\|}{\|\nabla X(p)^{*}X(p)\| \| \nabla X(p)^{-1}X(p)\|} \geq  \frac{1}{\|\nabla X(p)^{*}\| \| \nabla X(p)^{-1}\|}=\frac{1}{\mbox{cond}\, (\nabla X({p}))}.    
$$
Hence, taking into account that  $0\leq \theta < 1/{\mbox{cond}\, (\nabla X({\bar p}))}$, we conclude from the last inequality that 
$$
 \lim_{p\to {\bar p}} \frac{\langle \grad\,\varphi(p), -v \rangle}{\|\grad\, \varphi(p)\| \| v\| }   \geq  \frac{1}{\mbox{cond}(\nabla X({\bar p}))}>\theta.
$$
Therefore, there exists  ${\hat{\delta}}<{{\bar \delta}}$ such that  ${\langle \grad\,\varphi(p), -v \rangle/\|\grad\, \varphi(p)\| \| v\| } >\theta$, for all $p\in B_{\hat{\delta}}(\bar{p})$, and  \eqref{eq:AdTesteNew} also holds for $X(p)\neq 0$. 
\end{proof}
Next  lemma shows, in particular, that for $k$ sufficiently large $\alpha_k\equiv 1$  given by \eqref{eq:armijo}. Consequently,  \eqref{eq:NM1233} becomes the  Newton iteration \eqref{eq:NM12PURO}. 
\begin{lemma}\label{Le:SupLinConPhi}
Let  $\bar{p}\in\mathbb{M}$ such that $X(\bar p)=0$.  If $\nabla X$ is continuous at $\bar p$ and $\nabla X(\bar p)$ is nonsingular,  then  there exists $0<\hat{\delta}<\delta_{\bar{p}}$ such that $B_{\hat{\delta}}(\bar{p})\subset \Omega$, $\nabla X(p)$ is nonsingular  for each $p\in B_{\hat{\delta}}(\bar{p})$ and 
\begin{equation}\label{eq:SupelMerit}
\lim_{p\to \bar p}\frac{\varphi\left(N_{R,X}(p)\right)}{\parallel X(p)\parallel^{2}}=0.
\end{equation}
As a consequence, there exists a $\delta>0$ such that,  for all  $\sigma\in(0,1/2)$ and $\delta< \hat{\delta}$  there holds
\begin{equation}\label{eq:ArmijoCond}
\varphi\left(N_{R,X}(p)\right) \leq \varphi\left((p)\right)+ \sigma \left\langle\grad\,\varphi(p), -\nabla X(p)^{-1}X(p)\right\rangle, \qquad  \forall~p\in B_{\delta}(\bar{p}).
\end{equation}
\end{lemma}
\begin{proof}
Using  Lemma~\ref{le:NonSing} we obtain that  there exists $0<\hat{\delta}<\delta_{\bar{p}}$ such that $B_{\hat{\delta}}(\bar{p})\subset \Omega$ and $\nabla X(p)$ is nonsingular  for each $p\in B_{\hat{\delta}}(\bar{p})$.  We proceed to prove \eqref{eq:SupelMerit}.  To simply the notation we define
\begin{equation}\nonumber%\label{eq:Nstep}
v_{p}=   -\nabla X(p)^{-1}X(p), \qquad p\in B_{\hat{\delta}}(\bar{p}).
\end{equation}
Since $X(\bar p)=0$  and $\nabla X$ is continuous at $\bar p$ and nonsingular  we have ${\lim_ {p\to \bar{p}}}v_{p}=0$. Moreover, since the parallel transport is an isometry and taking into account \eqref{eq:qwhj1} and \eqref{diferenciabilidade} we can conclude that
\begin{equation}\nonumber
\varphi\left(N_{R,X}(p)\right)=  \dfrac{1}{2}\parallel X\left(R_{p}v_p\right)-P_{\bar p R_{p}v}X(\bar p)\parallel^{2}\leq \frac{1}{2a_0^2}\left(\| \nabla X(\bar{p})\| + \| r(R_{p}v)\|\right)^2 d^2(R_{p}v, \bar{p}).
\end{equation}
Hence, after some simples  algebraic manipulations we can conclude  from the last inequality that 
\begin{equation}\label{sssss}
\dfrac{\varphi\left(N_{R,X}(p)\right)}{\parallel X(p)\parallel^{2}} \leq  \frac{1}{2a_0^2} \left(\| \nabla X(\bar{p})\| + \| r(R_{p}v_p)\|\right)^2 \dfrac{ d^2(R_{p}v, \bar{p})}{d^{2}(p, \bar p)}\dfrac{d^{2}(p, \bar p)}{\parallel X(p)\parallel^{2}}, \quad  \forall ~p\in B_{\hat{\delta}}(\bar{p})\backslash\{\bar{p}\}.
\end{equation}
On the other hand,  owing that  $X(\bar p)=0$ and $\nabla X(\bar p)$ is nonsingular, it is easy to see that 
\begin{equation}\label{eq:AAABBB}
\|R^{-1}_{\bar p}p\|\leq \left\| \nabla X(\bar p)^{-1}\left[P_{p\bar{p}}X(p)- X(\bar p)- \nabla X(\bar p)R^{-1}_{\bar p}p\right]\right\|+ \left\| \nabla X(\bar p)^{-1}P_{p\bar{p}}X(p)\right\|.
\end{equation}
Since $\nabla X(\bar{p})$ is nonsingular, using \eqref{eq:qwhj1} and  \eqref{diferenciabilidade}, and taking into account that $ {\lim_{p\to \bar{p}}}r(p)=0$, we can take  $\bar{\delta}>0$ with $0<\bar{\delta}<\delta_{\bar{p}}$ such that
\begin{eqnarray}\nonumber
\left\|\nabla X(\bar{p})^{-1}\left[P_{p\bar{p}}X\left(p\right)- X({\bar p})- \nabla X(\bar p)R^{-1}_{\bar{p}}p\right]\right\|&\leq& \dfrac{a_{0}}{2a_{1}}\left\|R^{-1}_{\bar{p}}p\right\|\\\nonumber
&\leq &\dfrac{1}{2a_{1}}d\left(p, \bar{p}\right) , \qquad  \forall ~p\in B_{\bar{\delta}}(\bar{p}).
\end{eqnarray}
Thus, using \eqref{eq:AAABBB} and  \eqref{eq:qwhj1} we conclude that  $d(p, \bar{p})\leq   d(p, \bar{p})/2  + a_{1}\|\nabla X(\bar{p})^{-1}P_{p\bar{p}}X\left(p\right)\|$,  for all $p\in B_{\bar{\delta}}(\bar{p})$, which is equivalent to 
$$
\frac{d\left(\bar{p},p\right)}{|X\left(p\right)\|}\leq  2a_{1} \|\nabla X(\bar{p})^{-1}\|, \qquad  \forall ~p\in B_{\bar{\delta}}(\bar{p}).
$$
Letting $ \tilde{\delta}=\min\{ \hat{\delta}, \bar {\delta} \}$ we conclude from \eqref{sssss}  and last inequality that,   all $p\in B_{\tilde{\delta}}(\bar{p})\backslash\{\bar{p}\}$,  holds
$$
\dfrac{\varphi\left(N_{R,X}(p)\right)}{\parallel X(p)\parallel^{2}} \leq  \frac{2a_1^2}{a^2_0}\|\nabla X(\bar{p})^{-1}\|^{2}\left( \| \nabla X(\bar{p})\| + \| r(R_{p}v_p)\|\right)^2 \dfrac{ d^2(R_{p}v, \bar{p})}{d^{2}(p, \bar{p})}.  
$$
Therefore, using Lemma~\ref{L5} and considering $\lim_{p\to \bar{p}}r(R_{p}v_p)=0$, the equality  \eqref{eq:SupelMerit}  follows by taking limit, as $p$ goes to $\bar{p}$,  in the latter  inequality.  For proving \eqref{eq:ArmijoCond}, we first use \eqref{eq:SupelMerit} for concluding that  there exists a $\delta>0$  such that,  $\delta< \hat{\delta}$ and for   $\sigma \in (0,1/2)$ we have
$$
\varphi\left(N_{R,X}(p)\right) \leq  \dfrac{1-2\sigma}{2}\parallel X(p)\parallel^{2},\quad \forall~ p\in B_{\delta}(\bar{p}).
$$
Since    $\grad\,\varphi(p)= \nabla X(p)^{*}X(p)$ we obtain $\left\langle\grad\,\varphi(p),  -\nabla X(p)^{-1}X(p)\right\rangle= -\| X(p)\|^{2}$, then the  last inequality  is equivalent to \eqref{eq:ArmijoCond}  and the proof is concluded.
\end{proof}
%%%%%%%%%%%%%%%%%%%%%%%
\subsection{Convergence Analysis}
%%%%%%%%%%%%%%%%%%%%%%%
In this section we  establish our  main results, namely, the global convergence and superliner rate  of the sequence generated by Algorithm~\ref{AL:dampedRetraction2}.   Before presenting  these results,  we remark that  the  well-definition of  the sequence generated by Algorithm~\ref{AL:dampedRetraction2}  follows from Lemma~\ref{L:DescDir},  see  \cite[Lemma 6]{Bortoloti2020}.  It is worth noting  that,   if the sequence generated by Algorithm~\ref{AL:dampedRetraction2} is finite, then the last point generated is a solution of \eqref{eq:TheProblem0} or  it is critical point of  $\varphi$ defined in \eqref{eq:naturalfuncaomerito}. Thus, from now on,  we  assume that  $\{p_k\}$  is infinite. In this case, we have     $v_k\neq 0$ and $X(p_k)\neq 0$,  for all $k=0, 1, \ldots$.
\begin{theorem}\label{theoremAuxiliary123}
Let $\mathbb{M}$ be a Riemannian manifold,  $\Omega\subset\mathbb{M}$ be an open set and  $X:\Omega\to T\mathbb{M}$ be a continuously differentiable vector field.  Take  $R$ a retraction in $\mathbb{M}$. If $\bar p\in \Omega$ is an accumulation point of a sequence $\left\{p_{k}\right\}$ generated by Algorithm~\ref{AL:dampedRetraction2} then $\bar p$ is a critical point of $\varphi$. Moreover, assuming that $\nabla X$ is nonsingular at $\bar p$, the convergence of $\left\{p_{k}\right\}$ to $\bar p$ is superlinear and $X(\bar p)=0$.
\end{theorem}
\begin{proof}
Assume that   $\left\{p_{k}\right\}$ generated by  Algorithm~\ref{AL:dampedRetraction2} has  an accumulation point $\bar p$. First, we show that $\bar p$  is a critical point of $\varphi$.  We can assume $\grad\,\varphi(p_{k})\neq0$ for all $k=0,\,1,\,\ldots$ Hence, from \eqref{eq:armijo}
$$
\varphi(p_{k})-\varphi(p_{k+1})\geq-\sigma\alpha_{k}\langle\grad\,\varphi(p_{k}), v_{k}\rangle
 = \begin{cases}
\sigma\alpha_{k} \| X(p_k)\|^{2}>0, & \text{if} ~v_{k} ~\text{satisfies}~ \eqref{eq:Newton}\, and ~\eqref{eq:NewtonDirectionCondction};  \\
\sigma\alpha_{k} \|\grad\,\varphi(p_{k})\|^{2}>0, & \text{else}.
\end{cases}
$$
Thus, $\left\{\varphi(p_{k})\right\}$ is   strictly decreasing. Since  $\varphi$  is bounded  from below by zero,   it converges. Therefore,
$$
 0=\lim_{k\to \infty}  \alpha_{k}\left\langle \grad\,\varphi(p_{k}), v_{k}\right\rangle=\begin{cases}
\sigma  \lim_{k\to \infty}   \alpha_{k} \| X(p_k)\|^{2}= 0, & \text{if} ~v_{k} ~\text{satisfies}~ \eqref{eq:Newton}\, and ~\eqref{eq:NewtonDirectionCondction};  \\
\sigma  \lim_{k\to \infty}   \alpha_{k} \|\grad\,\varphi(p_{k})\|^{2}=0, & \text{else}.
\end{cases}
$$
In this  equality we have two possibilities, namely,  ${\lim\inf}_{k\to \infty} \alpha_{k}>0$ and  ${\lim\inf}_{k\to \infty} \alpha_{k}=0$.  First,  we assume that ${\lim\inf}_{k\to \infty} \alpha_{k}>0$.   Let  $\{p_{k_j}\}$ be a subsequence of $\{p_{k}\}$ such that  $\lim_{j\to \infty}  p_{k_j}=\bar p$ and ${\lim}_{j\to \infty} \alpha_{k_j}={\bar \alpha}>0$.  Taking into account $\lim_{j\to +\infty}p_{k_j}=\bar p$ and $X$ is continuous at $\bar p$  we conclude $X(\bar p)=0$ or $\grad\,\varphi(\bar p)=0$. Since $\grad\,\varphi(\bar p)=\nabla X(\bar p)^{*}X(\bar p)$,  $\bar p$ is a critical point of $\varphi$. Now, we assume that  ${\lim\inf}_{j\to \infty} \alpha_{k_{j}}=0$. We analyze the following  two possibilities: the sequence $\{v_{k}\}$ is unbounded or bounded.  Firs we assume $v_{k}$ is unbounded.  Since  $\lim_{j\to +\infty}p_{k_j}=\bar p$ and $X$ is continuous $\bar p$,  the direction  $v_{k}$  satisfies  \eqref{eq:DirectionOfGradiente} just for finite indexes, otherwise ${\lim}_{j\to \infty} v_{k_j}=-{\lim}_{j\to \infty} \nabla X(p_{k_j})^{*}X(p_{k_j})= -\nabla X({\bar p})^{*}X({\bar p})$, which contradicts the   assumption that  $v_{k}$ is unbounded.  Hence, we can assume  without loss of generality that $v_{k}$   satisfies  \eqref{eq:Newton} and \eqref{eq:NewtonDirectionCondction}, for all $k=0, 1, \dots$ Thus, for all $j$, it follows from \eqref{eq:NewtonDirectionCondction} that 
$$
0\leq \theta\|\grad\,\varphi(p_{k_j})\|\leq\dfrac{\|X({p_{k_j})}\|^{2}}{\|v_{k_j}\|}
.$$
Taking the limit as $j$ goes to infinity  in the last inequality, considering $v_{k}$ is unbounded and $\lim_{j\to \infty} \| X(p_k)\|^{2}=\|X({\bar p})\|$ we conclude that  $\lim_{j\to \infty} \| \grad\,\varphi(p_{k_j})\|=0$. Hence, $ \grad\,\varphi(\bar p)=0$. Now, we assume that  $\{v_{k}\}$ is bounded. In this case, we can assume that $\{v_{k}\}$ converges to some $\hat{v}$,  taking a subsequence if necessary.  Hence, we have $X({\bar p})+\nabla X({\bar p}){\hat v}=0$ or  ${\hat v}=-\nabla X({\bar p})^{*}X({\bar p})$, which implies that 
\begin{equation} \label{eq:retad}
\left\langle \grad\,\varphi(\bar p), ~ {\hat v} \right\rangle=\begin{cases}
-\| X({\bar p})\|^{2}, & \text{if} ~X({\bar p})+\nabla X({\bar p}){\hat v}=0;  \\
 -\|\grad\,\varphi({\bar p})\|^{2}, &  \text{if} ~ {\hat v}=-\nabla X({\bar p})^{*}X({\bar p}).
\end{cases}
\end{equation}
Since ${\lim\inf}_{j\to \infty} \alpha_{k_{j}}=0$, given $s\in\mathbb{N}$ we can take $j$ large enough such that $\alpha_{k_{j}}<2^{-s}$. Thus,    $2^{-s}$ does not satisfies  the Armijo's condition \eqref{eq:armijo}, i.e., 
$$
\varphi\left(R_{p_{k_{j}}}(2^{-s} v_{k})\right)> \varphi(p_{k_{j}})+\sigma 2^{-s}\left\langle \grad\,\varphi(p_{k_{j}}), ~ v_{{k_{j}}} \right\rangle.
$$
Since  $\lim_{j\to +\infty}v_{k_j}={\hat v}$,   $\lim_{j\to +\infty}p_{k_j}=\bar p$ and  the retraction  mapping is continuous, taking limit in the last inequality  we have   $\varphi\left(R_{\bar p}(2^{-s} {\hat v})\right) \geq \varphi(\bar p)+\sigma 2^{-s}\left\langle \grad\,\varphi(\bar p), ~ {\hat v} \right\rangle$,
which implies 
$ 
[\varphi\left(R_{\bar p}(2^{-s} {\hat v})\right)- \varphi(\bar p)]/2^{-s}\geq \sigma \left\langle \grad\,\varphi(\bar p), ~ {\hat v} \right\rangle.
$
 Then, letting $s$  goes to infinity  we conclude that $\left\langle \grad\,\varphi(\bar p), ~ {\hat v} \right\rangle  \geq \sigma \left\langle \grad\,\varphi(\bar p), ~ {\hat v} \right\rangle$, or equivalently $0\geq(\sigma-1) \left\langle \grad\,\varphi(\bar p), ~ {\hat v} \right\rangle$. Since $0<\sigma<1/2$ we have $\left\langle \grad\,\varphi(\bar p), ~ {\hat v} \right\rangle\geq0$. Taking into account $\left\langle \grad\,\varphi(p_{k_{j}}), ~ v_{{k_{j}}} \right\rangle\leq0$ for any $j$, we conclude $\left\langle \grad\,\varphi(\bar p), ~ {\hat v} \right\rangle=0$. Therefore, considering that $\grad\,\varphi(\bar p)=\nabla X(\bar p)^{*}X(\bar p)$,   it follows from \eqref{eq:retad} that $\grad\,\varphi(\bar p)=0$.
From now on, we assume that $\nabla X(\bar p)$ is nonsingular.  Since  $0=\grad\,\varphi(\bar p)=\nabla X(\bar p)^{*}X(\bar p)$,  we conclude that   $X(\bar{p})=0$. We proceed to prove  superlinearly convergence of  $\{p^k\}$  to ${\bar p}$. For that, we first will prove  there exists  $k_{0}$ such that $\alpha_{k}\equiv1$,  for all $k\geq k_{0}$.  Using    Lemma~\ref{le:NonSing} and Lemma~\ref{L:InvNx} we conclude   that there exists   $\hat{\delta}>0$ such that $\nabla X(p)$ is nonsingular  and $N_{R,X}(p)\in B_{{\delta}}(\bar p)$  for all $p\in B_{{\delta}}(\bar p)$ and all  $\delta\leq \hat{\delta}$.   Thus,  due to $\bar p$ be    an accumulation point of    $\left\{p_{k}\right\}$,  there exits $k_{0}$ such that $p_{k_{0}} \in B_{\hat{\delta}}(\bar p)$ and   shrinking  $\hat{\delta}$, if necessary,  from   Lemma~\ref{Le:SupLinConPhi} we have 
$$
\varphi\left(N_{R,X}(p_{k_{0}})\right) \leq \varphi\left(p_{k_{0}})\right)+\sigma\left\langle\grad\,\varphi(p_{k_{0}}), v_{k_{0}}\right\rangle, 
$$
where  $v_{k_{0}}=-\nabla X(p_{k_{0}})^{-1}X(p_{k_{0}})$. Hence,   \eqref{Newton's iterate mapping} and  \eqref{eq:armijo} imply   $\alpha_{k_{0}}=1$ and then   using   \eqref{eq:NM1233}  we conclude $p_{k_{0}+1}= N_{R,X}(p_{k_{0}})$. Due to  $N_{R,X}(p_{k_{0}}) \in B_{\hat{\delta}}(\bar p)$  we also have  $ p_{k_{0}+1}\in B_{\hat{\delta}}(\bar p)$.  Then,  an  induction step is completely analogous,   yielding 
\begin{equation}\nonumber\label{indicesdentrodabaciarapida123}
\alpha_{k}\equiv 1, \qquad    p_{k+1}= N_{R,X}(p_{k}) \in B_{\hat{\delta}}(\bar p), \qquad \forall ~k\geq k_{0}.
\end{equation}
 Finally,  to obtain the superlinear  convergence of $\left\{p_{k}\right\}$ to ${\bar p}$, let $\bar\delta >0$ be  given by  Theorem~\ref{Theorem of Newton}. Thus,  making $\hat{\delta}$ smaller if necessary so that $\hat{\delta}<\bar\delta$,  we can apply Theorem~\ref{Theorem of Newton} to conclude  that  $\left\{p_{k}\right\}$ converges  superlinearly to  $\bar p$.
\end{proof}
 %%%%%%%%%%%%%%%%%%%%%%%%%
\section{Numerical Experiments}\label{Sec:Numerical_Experiment} 
%%%%%%%%%%%%%%%%%%%%%%%%%
In this section, some examples are presented in order to examine the numerical behavior of algorithms studied in the previous sections. The  examples are established  on  spheres,   product of two Stiefel manifolds and cones of symmetric positive definite matrices. All numerical experiments have been developed by using MATLAB R2015b and were performed on an Intel Core Duo Processor 2.26 GHz, 4 GB of RAM, and OSX operating system. We have been considered convergence at the iterate  $k$ when $p_k\in \mathbb{M}$ satisfies $\|\mbox{grad}\,f(p_{k})\|<10^{-6}$, where $\|\cdot\|$ is the norm associated to metric of the considered manifold. The algorithms were interrupted when the step length reached a value less than $10^{-10}$ or the maximum number of $2000$ iterates  was reached. Moreover, we have been assumed $\sigma=10^{-3}$. All codes are freely available at \url{http://www2.uesb.br/professor/mbortoloti/wp-content/uploads/2020/08/public_codes.zip}

%%%%%%%%%%%%%%%%%%%%%%%%%
\subsection{Problems  on the sphere}
The aim of this section is to present problems on the sphere  $\mathbb{M}:=\left({\mathbb S}^{n}, \langle \, , \, \rangle\right)$, where ${\mathbb S}^{n}:=\left\{p:=(p_{1}, ..., p_{n+1})\in\mathbb{R}^{n+1}:\parallel p\parallel=1\right\}$ is endowed with the Euclidian inner product $\langle \, ~,~ \, \rangle$ and its corresponding norm $\| \cdot \|$. The tangent hyperplane in $M$  at $p$ is given by $T_{p}\mathbb{M}:=\left\{v\in\mathbb{R}^{n+1}:\langle p, v\rangle=0\right\}$.  It is worth noting that  problems have already studied in \cite{Bortoloti2020}. They were solved by using Algorithm~\ref{AL:jogo_with_retraction} where  the exponential
\begin{equation}\label{Eq:exp_Esfera}
\exp_{p}v= \cos(\parallel v\parallel)p+\sin(\parallel v\parallel)\dfrac{v}{\parallel v\parallel}, \qquad p\in{\mathbb S}^{n},\, v\in  T_{p}\mathbb{M}/\{0\}, 
\end{equation}
were used to update the iterates.  Here  we compare  the performance of Algorithm~\ref{AL:jogo_with_retraction}    by using  the exponential and  the following retraction
\begin{equation}\label{Eq:retraction_Esfera}
R_{p}v=\dfrac{p+v}{\parallel p+v\parallel}, \qquad p\in{\mathbb S}^{n},\, v\in  T_{p}\mathbb{M}/\{0\}.
\end{equation}
It is worth to point out that, for all problems studied in this section,  all results obtained show that  Algorithm~\ref{AL:jogo_with_retraction}  with the retraction \eqref{Eq:retraction_Esfera} has a better performance than the exponential \eqref{Eq:exp_Esfera}. We also note that retraction \eqref{Eq:retraction_Esfera} has  performed better than \eqref{Eq:exp_Esfera} in others algorithms, see for example \cite{Smith1994}. 
%%%%%%%%%%%%%%%%%%%%%%%%%%%%%%%%
\subsubsection{ Non conservative vector field}
In this section, we consider the problem of finding a singularity of  a non conservative vector field on the the sphere.  For that  fix a point  ${\bar p}\in \mathbb{M}$ and  $Q$ a $n\times n$ skew-symmetric matrix, and define  the vector field   $X$ by 
\begin{equation}\nonumber%\label{Eq:vector_Non_Conservative}
X(p)=Q(p-{\bar p})-\left\langle p, Q(p-{\bar p})  \right\rangle p.  
\end{equation}
 Note  that ${\bar p}$ is a singularity of $X$, i.e., $X( {\bar p})=0$.  We recall the  covariant derivative of $X$ is given by
\begin{equation}\nonumber%\label{eq:covar_nonconsevative}
\nabla X(p)=\left[I+pp^T\right]Q-p\left(Q(p-{\bar p})\right)^T -\left\langle p, Q(p-{\bar p}) \right\rangle I.
\end{equation}
Since $Q$ is not symmetric neither is $\nabla X(p)$, consequently, $X$ is not a conservative vector field.
\begin{table}[h]
	\centering
	\begin{tabular}{llllll}
		\toprule
		\multirow{2}{*}{n} &
		\multicolumn{2}{c}{Iter} &
		\multicolumn{2}{c}{E$X$} \\
		& {$\exp_{p}v$} & {$R_{p}v$}  & {$\exp_{p}v$} & {$R_{p}v$} \\
		\midrule
		2    & 7  & 7 & 14 & 14 \\
		50   & 22 & 21 &75 & 64 \\
		500  & 27 & 26 &  71 & 61 \\
		1000 & 15 & 12 & 31 & 25 \\
		\bottomrule
	\end{tabular}
\caption{Comparison between retractions \eqref{Eq:exp_Esfera} and \eqref{Eq:retraction_Esfera} for Algorithm \ref{AL:jogo_with_retraction}.}
\label{tab:sphere}
\end{table}
We present a comparative study between retractions \eqref{Eq:exp_Esfera} and \eqref{Eq:retraction_Esfera} for Algorithm  \ref{AL:jogo_with_retraction}. We have performed numerical experiments for dimensions $n=2,50,500,1000$. For each dimension, we consider a skew matrix, $Q$, defined as $Q = A - A^T$, where $A$ is a random matrix generated by code \verb|A = randn(n,n)|. We have taken one initial guess on sphere for dimension, in order to start algorithm. Our numerical results are presented in Table \ref{tab:sphere}, where Iter and EX denote the iteration numbers and evalutiuon of vector field, respectively. It can be seen that the exponential mapping presents a greater number of vector fields evaluations than retraction. Moreover, the retraction \eqref{Eq:retraction_Esfera} solves the problems with a smaller number of iterations than exponential map \eqref{Eq:exp_Esfera}.  We also remark that the quantities   Iter and EX do not depend on the dimension of the problem. 
%%%%%%%%%%%%%%%%%%%%%%%%%%%%%%
\subsubsection{Rayleigh quotient} \label{sec:rayleigh}
Numerical results to minimizer the Rayleigh quotient  on the sphere is presented in this section.  Let $A$  be  an $n\times n$ symmetric and  positive definite matrix and $f:{\mathbb S}^{n}\to\mathbb{R}$  be the Rayleigh quotient, 
\begin{equation}\nonumber%\label{Eq:RayQuotient}
f(p)=p^{T}Ap.
\end{equation}
 The  gradient and the hessian of $f$ are given by, respectively, by 
$$
\mbox{grad}\, f(p)= -2 \left[I+pp^T\right]Ap, \qquad \mbox{Hess}\,f(p)=\left[I+pp^T\right]\left[A- p^TApI\right].
$$

In order to develop our numerical experiments, we define the Rayleigh quotient  mapping by considering symmetric positive definite matrices $A$, given by the following codes
\begin{itemize}
	\item[1.] \verb|A = gallery('poisson',ceil(sqrt(n))+1)|,
	\item[2.] \verb|A = ones(n,n); A = A*diag(diag(A))+diag((2*n)*ones(n,1))|,
	\item[3.] \verb|e = ones(n,1); A = spdiags([e 10*e e], -1:1, n, n)|,
	\item[4.] \verb|v = rand(n,1); A = diag(v); A(1,n) = 1; A(n,1) = 1|,
	\item[5.] \verb|A = sprandsym(n,0.7,0.1,1)|.
\end{itemize}
For each type of matrix above, simulations were performed for dimensions $n=500$, $750$, $1000$, $1250$, $1500$, where we considered, for each one, 10 initial guesses randomly taken on sphere, totalizing 250 problems.

\begin{figure}[h!]
	\centering
	\includegraphics[scale=0.65]{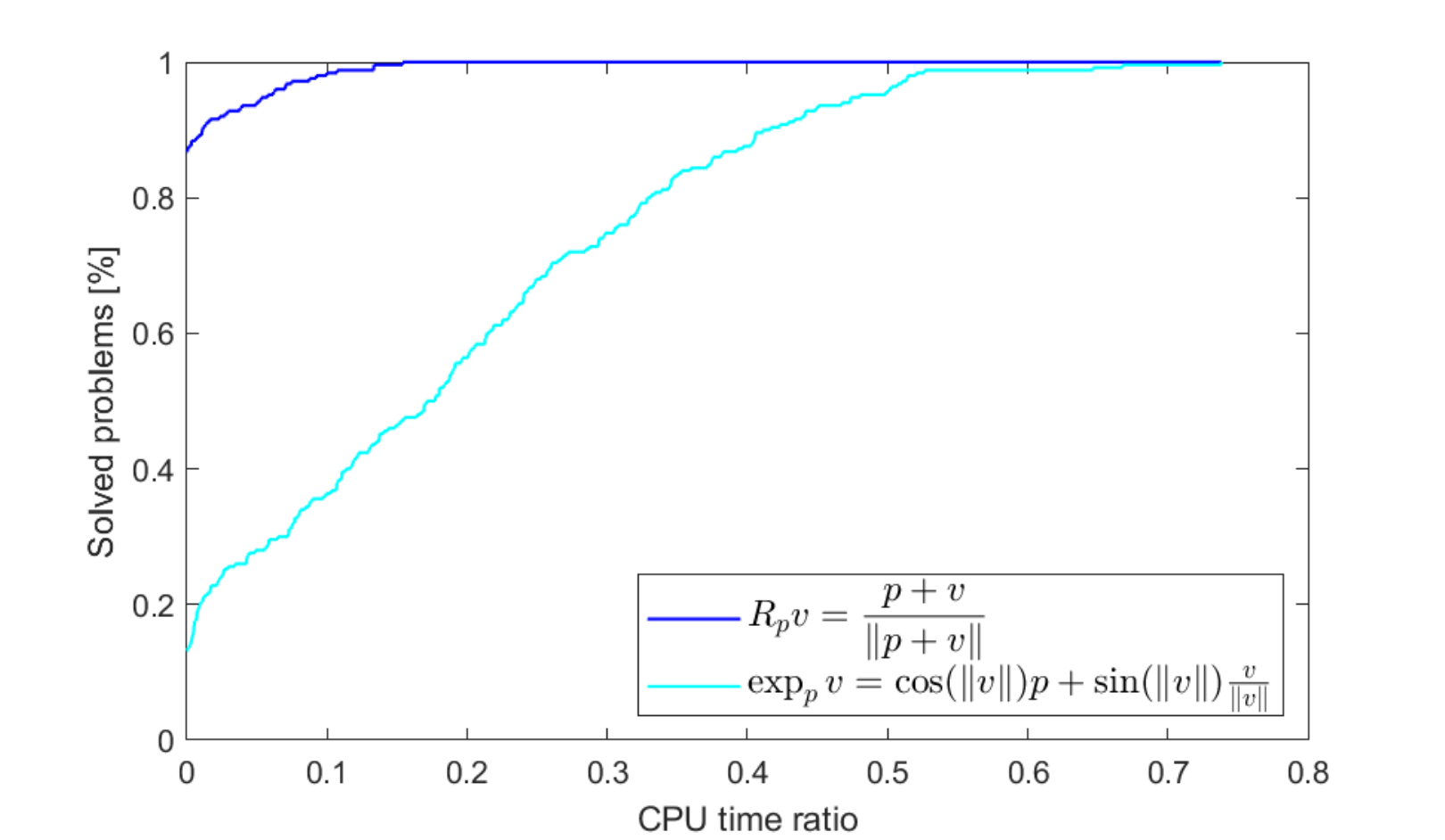}
	\caption{Performance profile for Algorithm \ref{AL:jogo_with_retraction} with retractions \eqref{Eq:exp_Esfera} and \eqref{Eq:retraction_Esfera}.}
	\label{grafico_Rayleigh}
\end{figure}

In figure \ref{grafico_Rayleigh} can be seen that the retraction \eqref{Eq:retraction_Esfera} solves all problems faster than exponential \eqref{Eq:exp_Esfera}. We can also observe that Algorithm \ref{AL:jogo_with_retraction} with \eqref{Eq:retraction_Esfera} solves all problems before 0.2 CPU time ratio. This behaviour can be assigned to the simple form of retraction \eqref{Eq:retraction_Esfera} so that fewer calculations are needed to generate the next point in the sequence, when considered the  classical exponential.

\subsection{The truncated singular value problem on Stiefel manifold} \label{sec:robust}
In this section, we study the globality of Algorithm \ref{AL:dampedRetraction2} to solve a truncated singular value problem. This problem, analyzed in \cite{Sato2013}, consists of 
\begin{align}\label{eq:TrunSingularProblem}
\begin{split}
\min F(P, Q)&=\mbox{Tr}(-P^{T}AQN) 
\\
\text{subject to}\,\left(P, Q\right)&\in\,St(p,m)\times St(p,n)
\end{split}
\end{align}
where, $\mbox{Tr}\left(X\right)$ denotes the trace of $X$, $St(p,n)=\left\{ P\in\mathbb{R}^{n\times p}\, :\,P^{T}P=I_{p}\right\}$ is the Stiefel manifold endowed with the Frobenius metric, $A\in \mathbb{R}^{m\times n}$ with $m\geq n$ is a given matrix, and $N = \mbox{diag}\left(\mu_1, \mu_2,\ldots ,\mu_p\right)$ is a diagonal matrix with $\mu_1 > \mu_2 >\ldots >\mu_p > 0$ for $p\leq n$. Considering $P\in St(p, m)$, the tangent space of $St(p, m)$, at $P$, is given by $T_{P}St(p,m)=\left\{PB+P_{\perp}C\,:\,B\in \mbox{Skew}(p),\,C\in \mathbb{R}^{(m-p)\times p}\right\}$
where $\mbox{Skew}(p)$ denotes the set of all $p\times p$ skew-symmetric matrices and $P_{\perp}$ is an $n\times(n-p)$ orthonormal matrix such that $PP^{T}+ P_{\perp}P_{\perp}^{T}=I_{m}$. In order to implement the Algorithm \ref{AL:jogo_with_retraction} to solve problem \eqref{eq:TrunSingularProblem}, we present the Newton equation \eqref{eq:Newton1} and the safeguard direction \eqref{eq:DirectionOfGradiente}. The Newton equation at $(P, Q)\in St(p,m)\times St(p,n)$ in the variables $U,V\in T_{P}St(p,m)\times T_{Q}St(p,n) $ is writen as
\begin{equation}\label{eq:newton_truncado}
\left\{
\begin{array}{ccc}
VS_1-AUN-P \mbox{ sym}\left(P^{T}\left(VS_1-AUN\right)\right)&=&AQN-PS_1, \\
US_2-A^TVN-Q\mbox{ sym}\left(Q^T\left(US_2-A^TVN\right)\right)&=&A^TPN-QS_2, 
\end{array}
\right.
\end{equation}
see \cite{Aihara2017}.  The safeguard direction \eqref{eq:DirectionOfGradiente} is given by
\begin{multline*}
\grad\,\varphi\left(P, Q\right)\left(V, U\right)=\big( VS_1-AUN-P\mbox{sym}\left(P^T\left(VS_1-AUN\right)\right),\\
US_2-A^TVN-Q\mbox{sym}\left(Q^T\left(US_2-A^TVN\right)\right)\big),
\end{multline*}
where $\mbox{sym}\left(X\right)=\left(X+X^T\right)/2$ denotes the symmetric part of a square matrix $X$, $S_1=\mbox{sym}\left(P^TAQN\right)$ and $S_2=\mbox{sym}\left(Q^TA^TPN\right)$. To compute the iterates on Stiefel manifold, we have considered the following retractions:
\begin{itemize}
\item Exponential map, \cite{Edelman1998}, given by
\begin{equation}\label{eq:st_exp}
R_P(V) = PM+QN,
\end{equation}
where $QR$ is the compact decomposition of $(I_n-PP^T)V$ such that $Q$ is $n$-by-$p$, $R$ is $p$-by-$p$ and $M, N$  are $p$-by-$p$ matrices given by the $2p$-by-$2p$ matrix exponential

	\[
	\left( \begin{array}{c}
	M \\
	N
	\end{array} \right)
	= \exp
	\left(
	\begin{array}{cc}
	P^TV  & -R^T \\
	R & 0
	\end{array}
	\right)
	\left(
	\begin{array}{c}
	I_p \\
	0
	\end{array}
	\right).
	\]
\item Cayley map, \cite{wen2013feasible}, writen as
\begin{equation} \label{eq:st_cayley}
R_P(V)=P+M\Big(I_{2p}-\frac{1}{2} N^TM\Big)^{-1}M^TP,
\end{equation}
where $M=[\Pi_P V ,P]$ and $N=[P,-\Pi_P V]$, with $\Pi_P = I_p-PP^T/2$.
\item Polar map, \cite{Absil2009}, given by
\begin{equation}\label{eq:st_polar}
R_P(V) = (P+V)\Big(I_p+V^TV\Big)^{-1/2}.
\end{equation}
\item \mbox{qf} retraction, \cite{Absil2009}, writen as
\begin{equation}\label{eq:st_qf}
R_P(V) = \mbox{qf}(P+V),
\end{equation}
where $\mbox{qf}(P)$ denotes the $Q$ factor of the  QR decomposition with nononegative elements on the diagonal of the upper triangle matrix.
\end{itemize}

In order to solve the equations in \eqref{eq:newton_truncado} we used the ideas presented in \cite{Aihara2017}. The numerical experiments were performed on the product of two Stiefel manifolds, $St(p,m)\times St(p,n)$, with $(m,n,p)=(5,3,2),(7,5,2),(10,5,3), (20,10,3)$. For each $(m,n,p)$, the generated problem has  $(P^*,Q^*) \in St(p,m)\times St(p,n)$ as  critical point and  $A=P^*N{Q^*}^T$, where $N=\mbox{diag}(p, p-1,\ldots,1)$ and $P^*$, $Q^*$ are given by the Q factors of the QR decompositions of two random matrices.  It were taken $10$ initial guesses $(P_{0},Q_{0})$, for each dimensions $(m,n,p)$, by considering the matrices $P_0 = \mbox{qf}(P^*+\mbox{randn}(m,p) \, \varepsilon)$ and $Q_0 = \mbox{qf}(Q^*+\mbox{randn}(n,p) \, \varepsilon)$ with 
$\varepsilon=10^{-4}, 10^{-3},\cdots,10^{3}$.

In figures \ref{fig:aiharaexp}-\ref{fig:aiharaqf} we present the percentege of solved problems when the initial guesses are taken by considering the values from $\epsilon=10^{-4}$ to $\epsilon=10^{3}$. As it can seen, the Algorithm \ref{AL:dampedRetraction2} is not sensitive to increasing of $\epsilon$ solving $100\%$ of the problems for each retraction given in \eqref{eq:st_exp}-\eqref{eq:st_qf}. 
\begin{figure}[h!]
	\centering
	\subfigure[Exponential map given in \eqref{eq:st_exp}.]{\includegraphics[scale=0.25]{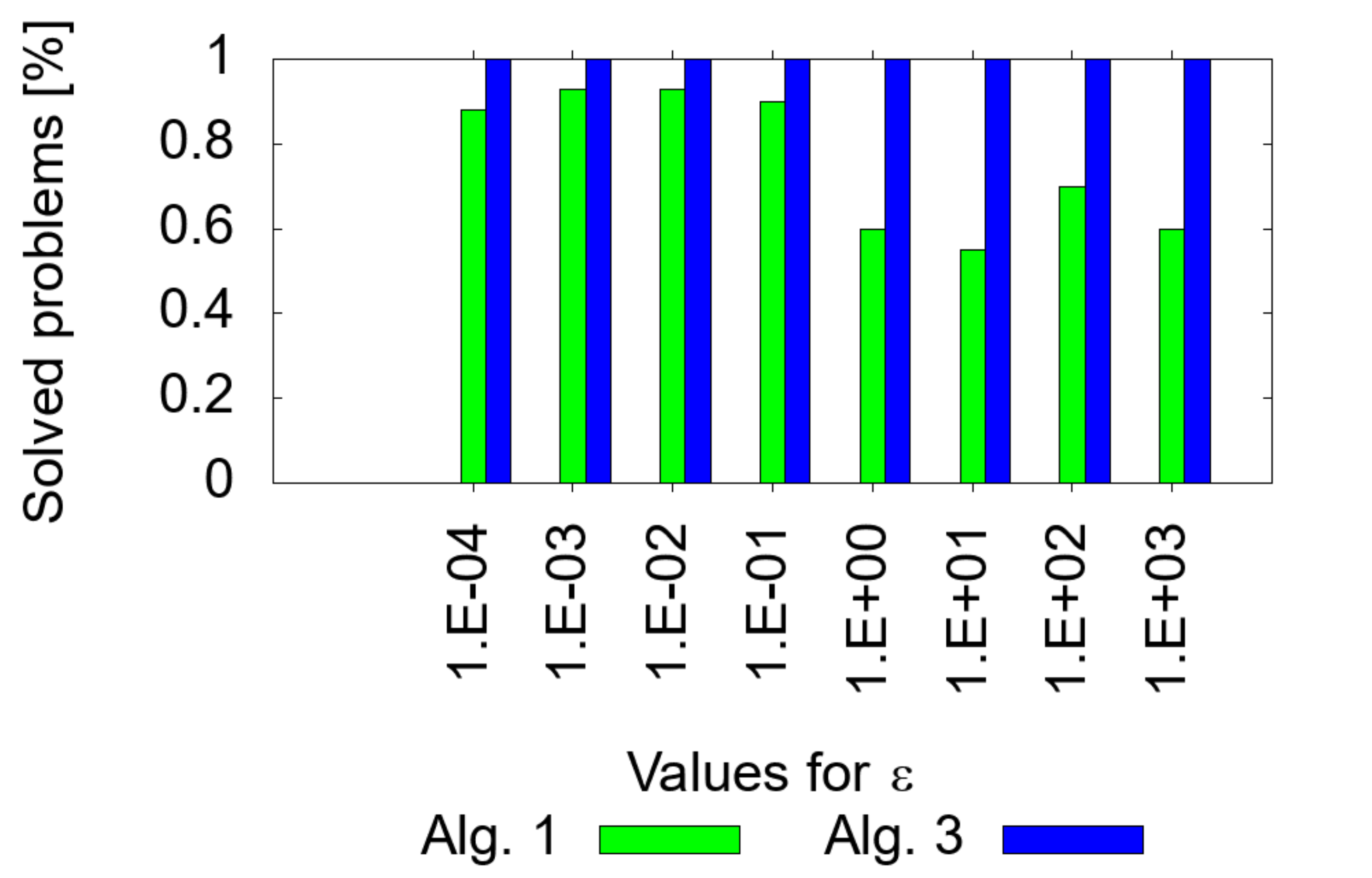}\label{fig:aiharaexp}}
	\subfigure[Cayley map given in  \eqref{eq:st_cayley}.]{\includegraphics[scale=0.25]{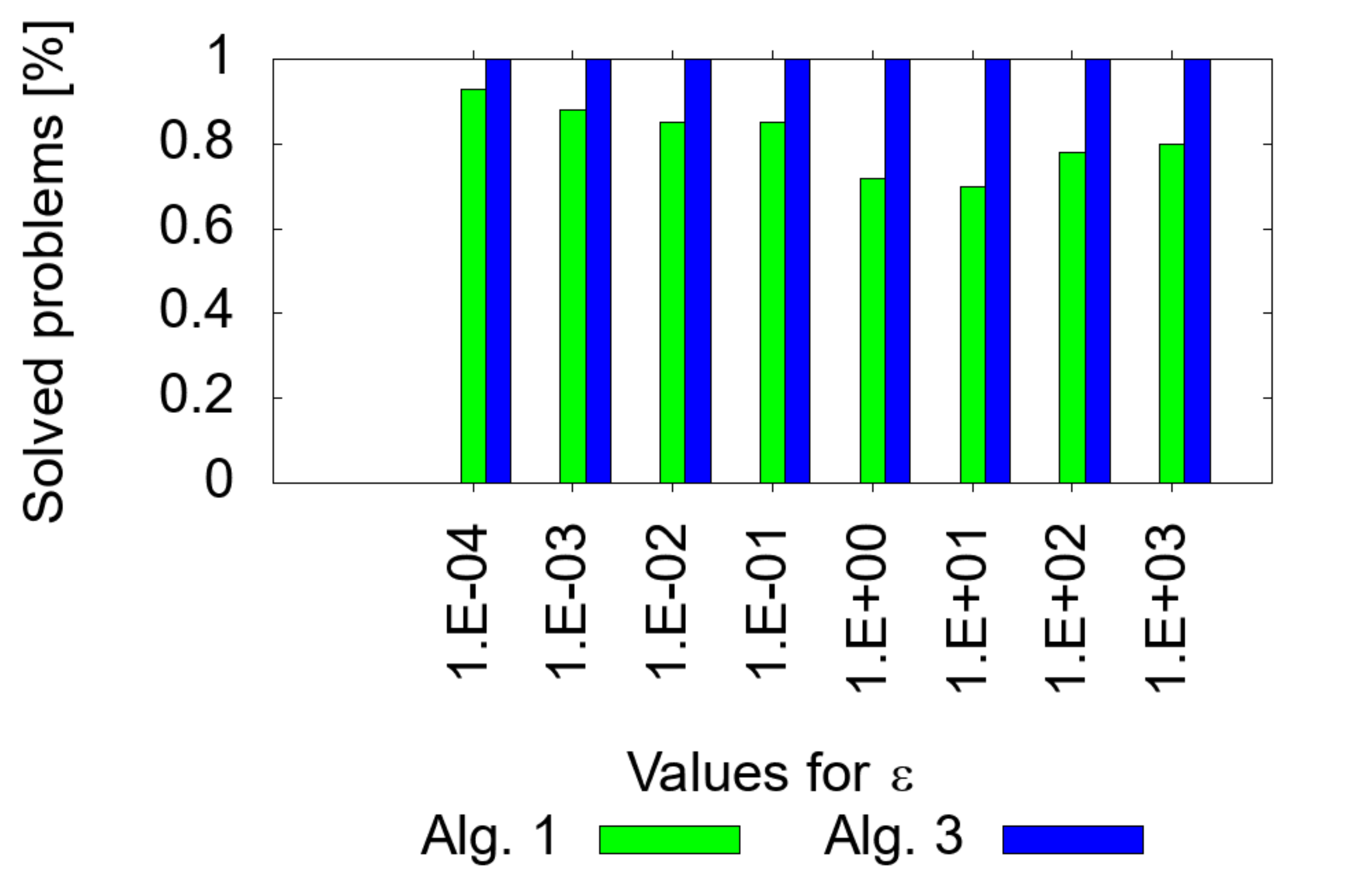}\label{fig:aiharacayley}}
	\\
	\subfigure[Polar map given in \eqref{eq:st_polar}.] {\includegraphics[scale=0.25]{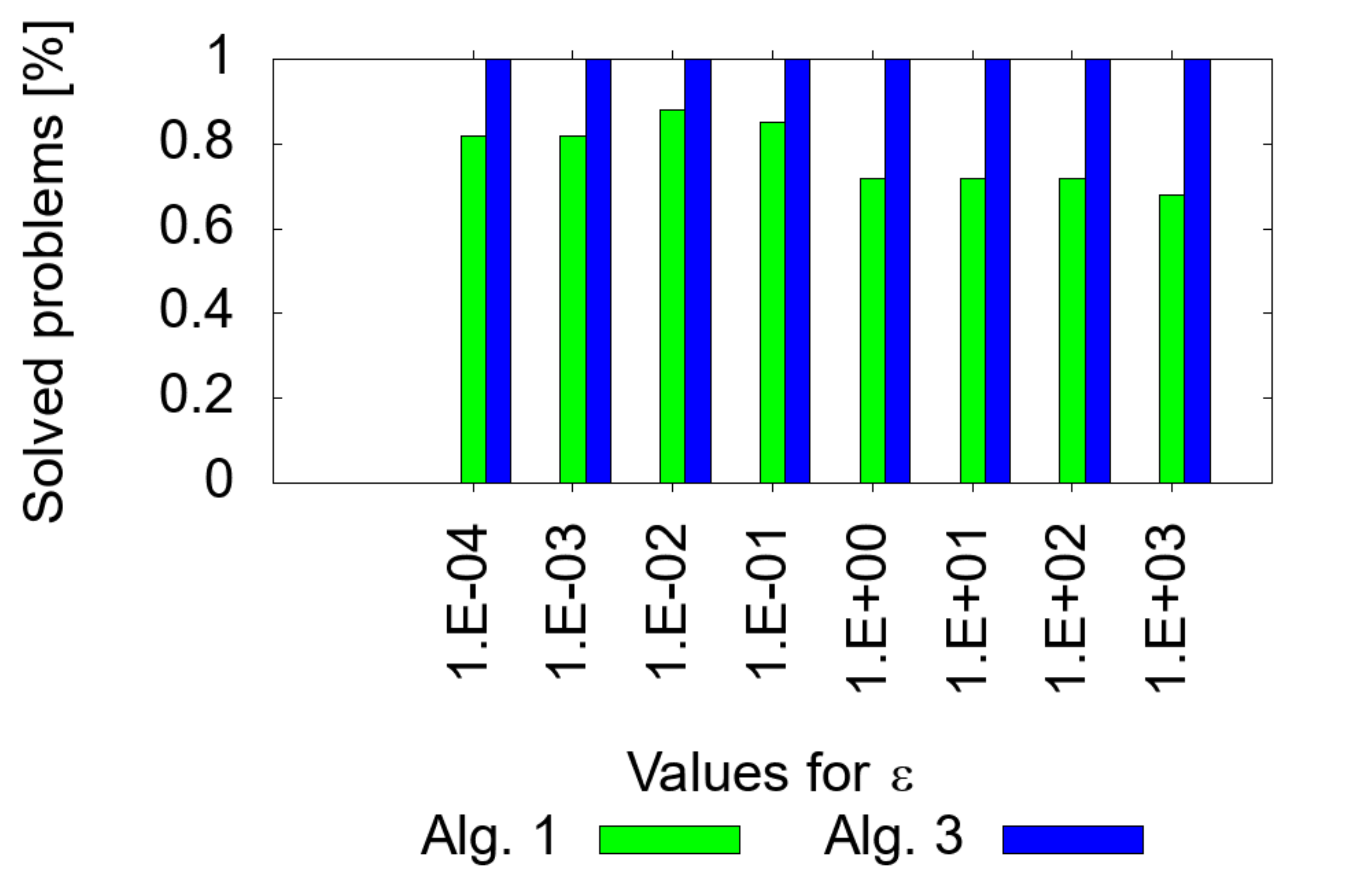} \label{fig:aiharapolar} }
	\subfigure[\mbox{qf} decomposition given in \eqref{eq:st_qf}.] {\includegraphics[scale=0.25]{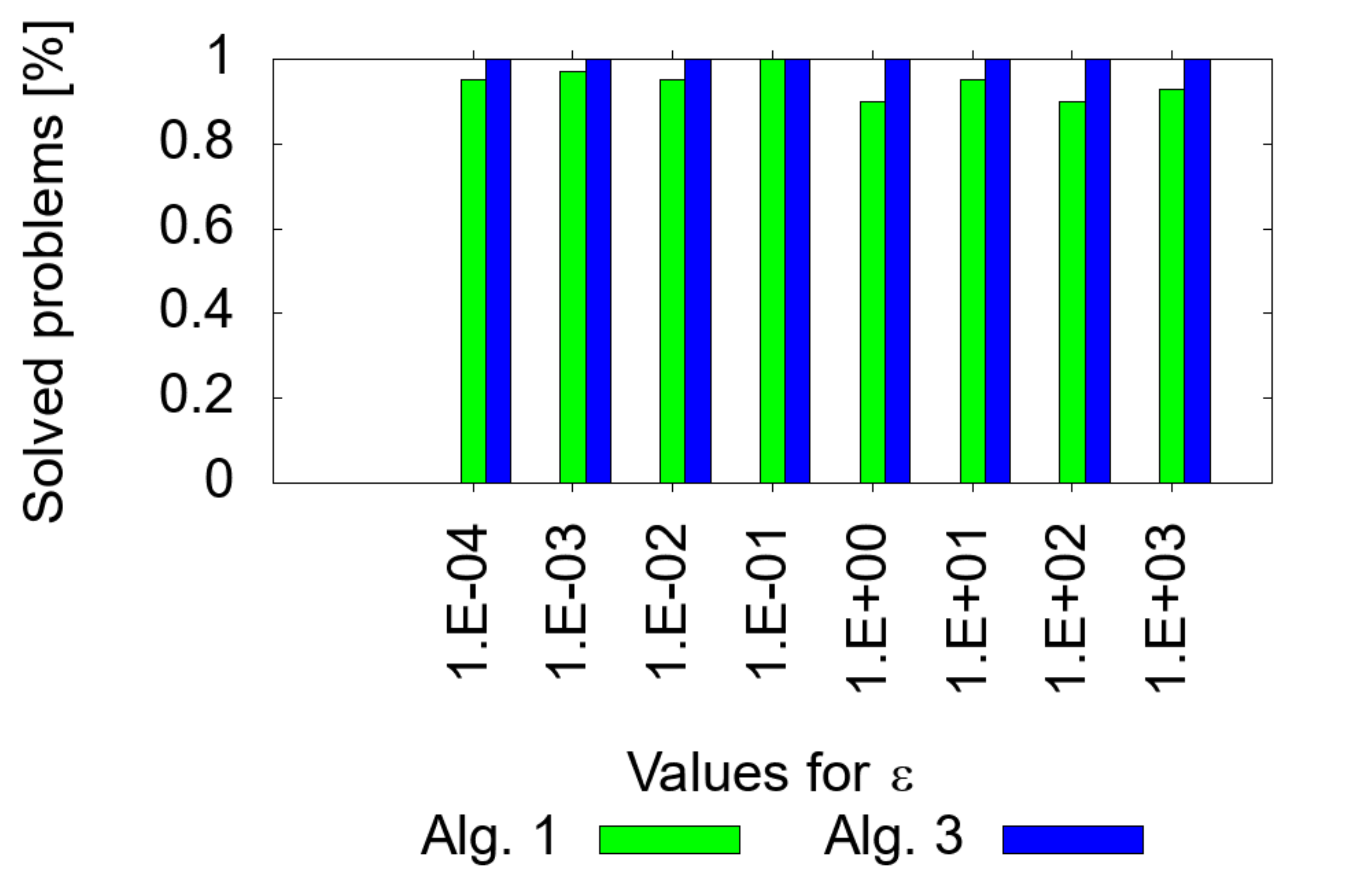} \label{fig:aiharaqf}}
	\caption{Globality analysis for Alg. \ref{AL:NewtonPuro} and Alg. \ref{AL:dampedRetraction2} considering retractions for some $\varepsilon$ values in case of problem \eqref{eq:TrunSingularProblem} and  $\theta=0.9$.}
\end{figure}

%It can be seen that as the $\varepsilon$ values are increased, Algorithm %\ref{AL:jogo_with_retraction} solves all problems.
%
%
%
%\begin{figure}[!htb]
%	\centering
%	\includegraphics[scale=0.8]{figures/dampedret_alldim.png}
%	\caption{Performance profile considering  DNM with EXP, Cayley, Polar and qf.}
%	%\label{grafico_Rayleigh}
%\end{figure}

\subsection{ Academic problem on the cone of symmetric positive definite matrices} \label{sec:rubustcone}
The aim of this section is to present some numerical experiments to illustrate the behavior  of Algorithms 2 and 3. For that, we have considered the problem for minimizing two different functions defined on the cone of  symmetric positive definite matrices, ${\mathbb P}^{n}_{++}$, endowed with the inner product given by $\langle U,V \rangle=\mbox{tr} (VP^{-1}UP^{-1}), P\in {\mathbb P}^{n}_{++}, U,V\in T_P{\mathbb P}^{n}_{++} \approx~\mathbb{P}^n$ where, $\mathbb{P}^n$ denotes the set of symmetric matrices. Consider $f_1,f_2:\mathbb{P}^{n}_{++}\to\mathbb{R}$ defined  by
\begin{equation}\label{eq:familyfuction}
f_1(P)=\ln\,\det P + \mbox{tr}\,P^{-1}, \quad \quad f_2(P)=\ln\det P - \mbox{tr}\,P,
\end{equation}
where $\det P$  and $\mbox{tr}\,P$ denote the determinant and  trace of $P$, respectively. The function $f_2$ can be related to robotics, see \cite{robotic2002,robotic2007,robotic2002_2}. It is worth noting that these problems have already studied in \cite{Bortoloti2020}. In order to implement the Algorithms~\ref{AL:jogo_with_retraction} and ~\ref{AL:dampedRetraction2} to minimize the functions in  \eqref{eq:familyfuction}, we  first present the Newton equation and the safeguard direction associated to them. The Newton equations for $f_1$ and $f_2$ are, respectively, given by 
\begin{equation}\nonumber
PV+VP=2\left(P´^2-P^3\right), \qquad P^{-1}V+VP^{-1}=2\left(P^{-1}-I\right),
\end{equation}
where $I$ denotes the $n\times n$ identity matrix. On the other hand, the gradient of merit function  for $f_1$ and $f_2$ can be, respectively, writen as
\begin{equation}\nonumber%\label{eq:gradmeritfunctionfamily}
\grad\,\varphi_{1}(P)=I -P^{-1}, \qquad  \qquad \grad\,\varphi_{2}(P)=P^{3}-P^{2}.
\end{equation}
For all $P\in \mathbb{P}^{n}_{++}$ and $V\in T_{P}\mathbb{P}^{n}_{++}$ we compute the iterates in \eqref{eq:NM12331} and \eqref{eq:NM1233} by using  the following  retractions
\begin{eqnarray}%\label{eq:Retracao1_coneSPD}
R_PV&=&P^{1/2} \exp \left(P^{-1/2}VP^{-1/2}\right) P^{1/2} , \label{eq:retexpcone} \\
R_PV&=&P\exp \left( P^{-1}V \right),  \label{eq:retsimcone} \\ 
R_PV&=&P+V+1/2VP^{-1}V, \label{eq:ret2ordercone}\\
R_PV&=&P+V.  \label{eq:ret1ordercone}
\end{eqnarray}
The map presented in \eqref{eq:retexpcone} is the classical exponential and \eqref{eq:retsimcone} is obtained from \eqref{eq:retexpcone},  see \cite{Hosseini2015}. For \eqref{eq:ret2ordercone} and  \eqref{eq:ret1ordercone}, see \cite{jeuris2012survey} and \cite{bini2014geometric}, respectively. 

For both  functions given by \eqref{eq:familyfuction} we consider dimensions $n=100, 200, \cdots, 1000$ and 5 initial guesses for each dimension, totalizing 50 problems for each retraction. The initial guesses were generated by considering the same type of symmetric positive definite matrices presented in Section~\ref{sec:rayleigh}. Figure \ref{fig:f1dampedexpxsim} shows a performance of Algorithm \ref{AL:jogo_with_retraction} with the retractions \eqref{eq:retexpcone} and \eqref{eq:retsimcone} for minimizing the function $f_1$. It can be seen that the retraction \eqref{eq:retsimcone}   is much better than the second one. This behaviour can be justified by the smaller number of operations required to generate each iterate of the method.
\begin{figure}[h!]
\centering
\includegraphics[scale=0.65]{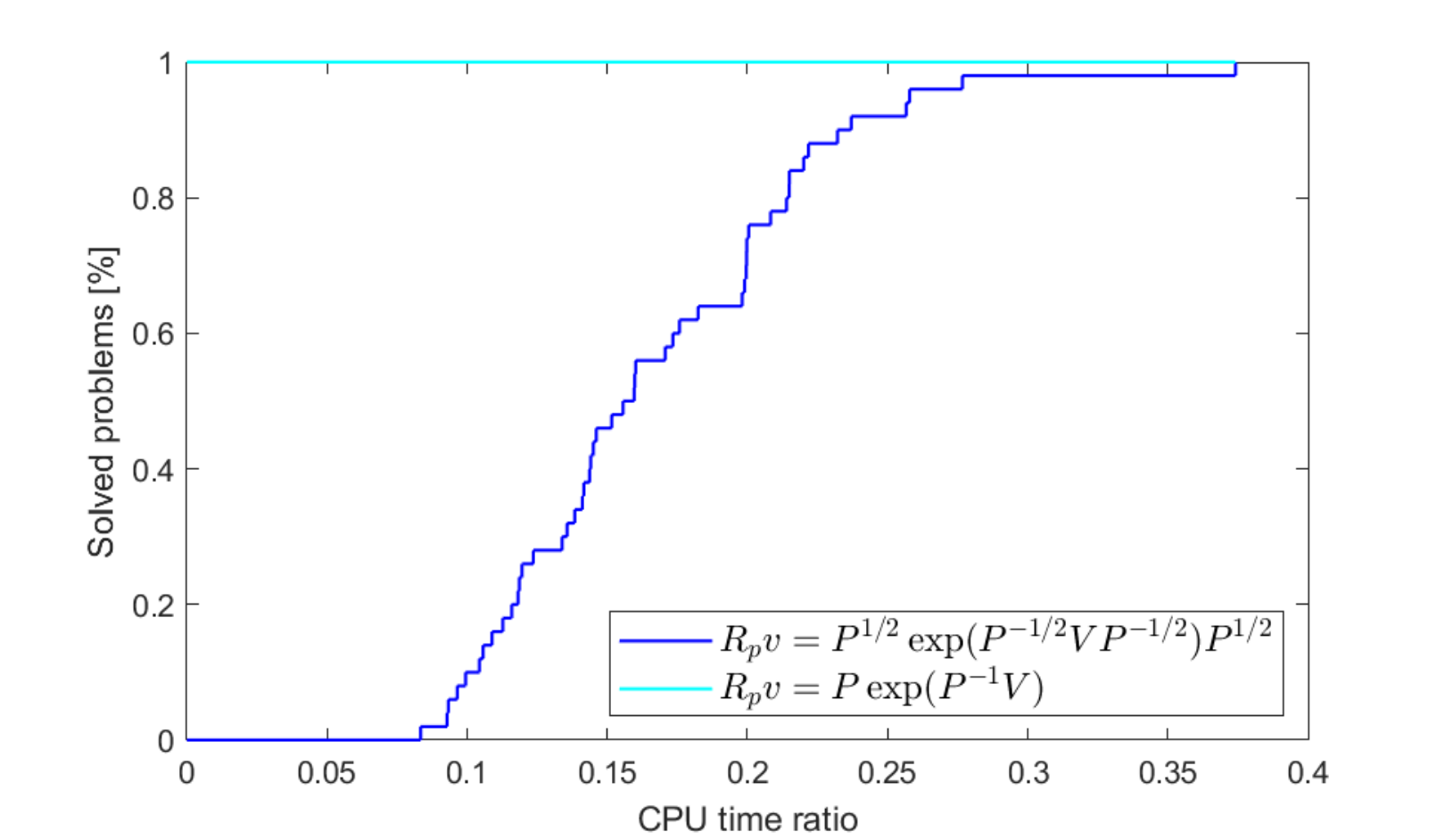}
\caption{Performance profiles comparing retractions \eqref{eq:retexpcone} and \eqref{eq:retsimcone} for $f_1$ in the Algorithm \ref{AL:jogo_with_retraction}.}
\label{fig:f1dampedexpxsim}
\end{figure}

In order to highlight the improviment  of the Algorithm \ref{AL:dampedRetraction2}, related to robustness,  we compare it with Algorithm \ref{AL:jogo_with_retraction} by using $\theta=0.9999$. It can be seen in figures \ref{fig:f11order} and \ref{fig:f12order} that Algorithm \ref{AL:jogo_with_retraction}, provided with retractions (34) and (35), did not solve all problems. This has happened because the step length has became too small. On the other hand, Algorithm \ref{AL:dampedRetraction2} was able to solve all problems because the inequality given by \eqref{eq:NewtonDirectionCondction} prevents an inappropriated decreasing of step length.

\begin{figure}[h!]
	\centering
	\subfigure[$R_PV=P+V$]{\includegraphics[scale=0.45]{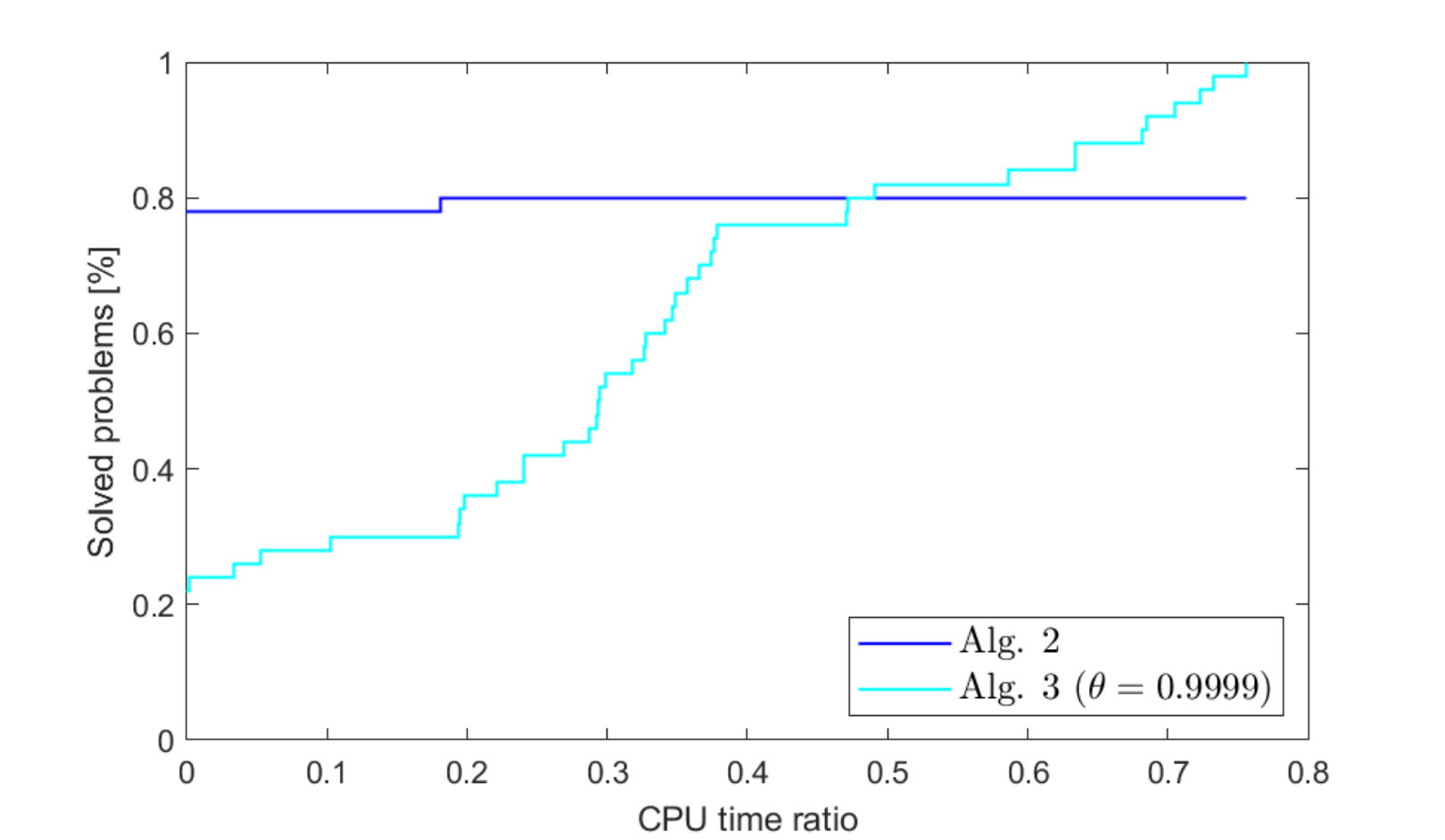}\label{fig:f11order}}
	\subfigure[$R_PV=P+V+1/2VP^{-1}V$]{\includegraphics[scale=0.45]{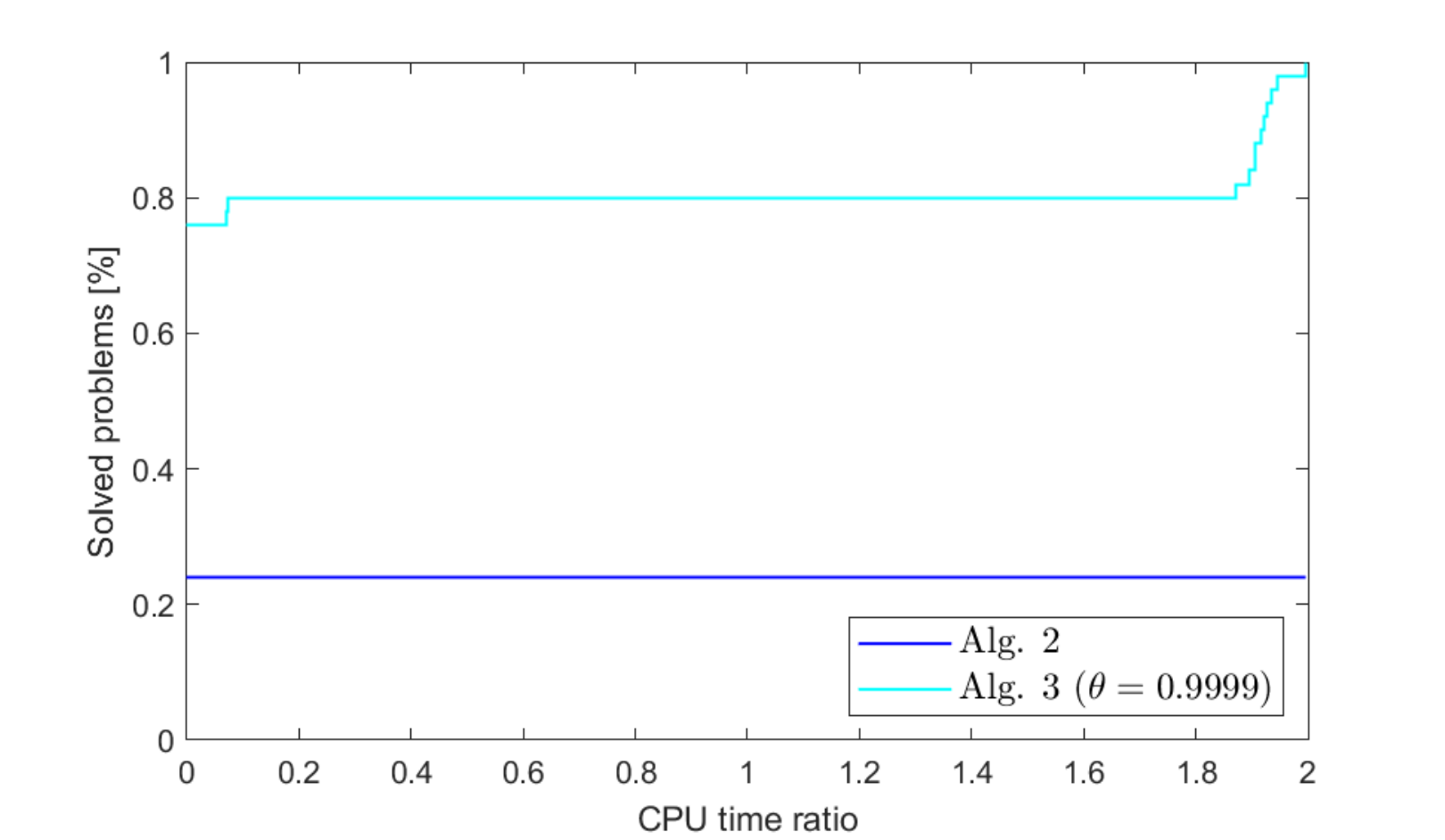}\label{fig:f12order}}
	\caption{Comparation  between Algorithms \ref{AL:jogo_with_retraction} and  \ref{AL:dampedRetraction2} with retractions \eqref{eq:ret1ordercone} and \eqref{eq:ret2ordercone} for $f_1$.}\label{fig:conef1}
\end{figure}

For minimizing the function $f_2$, we compare retractions \eqref{eq:retexpcone}-\eqref{eq:ret2ordercone} in the Algorithm \ref{AL:jogo_with_retraction} as can be seen in the figure \ref{fig:conef2}. It shows that the retraction \eqref{eq:ret1ordercone} is much better than the other.

\begin{figure}[h!]
	\centering
\includegraphics[scale=0.65]{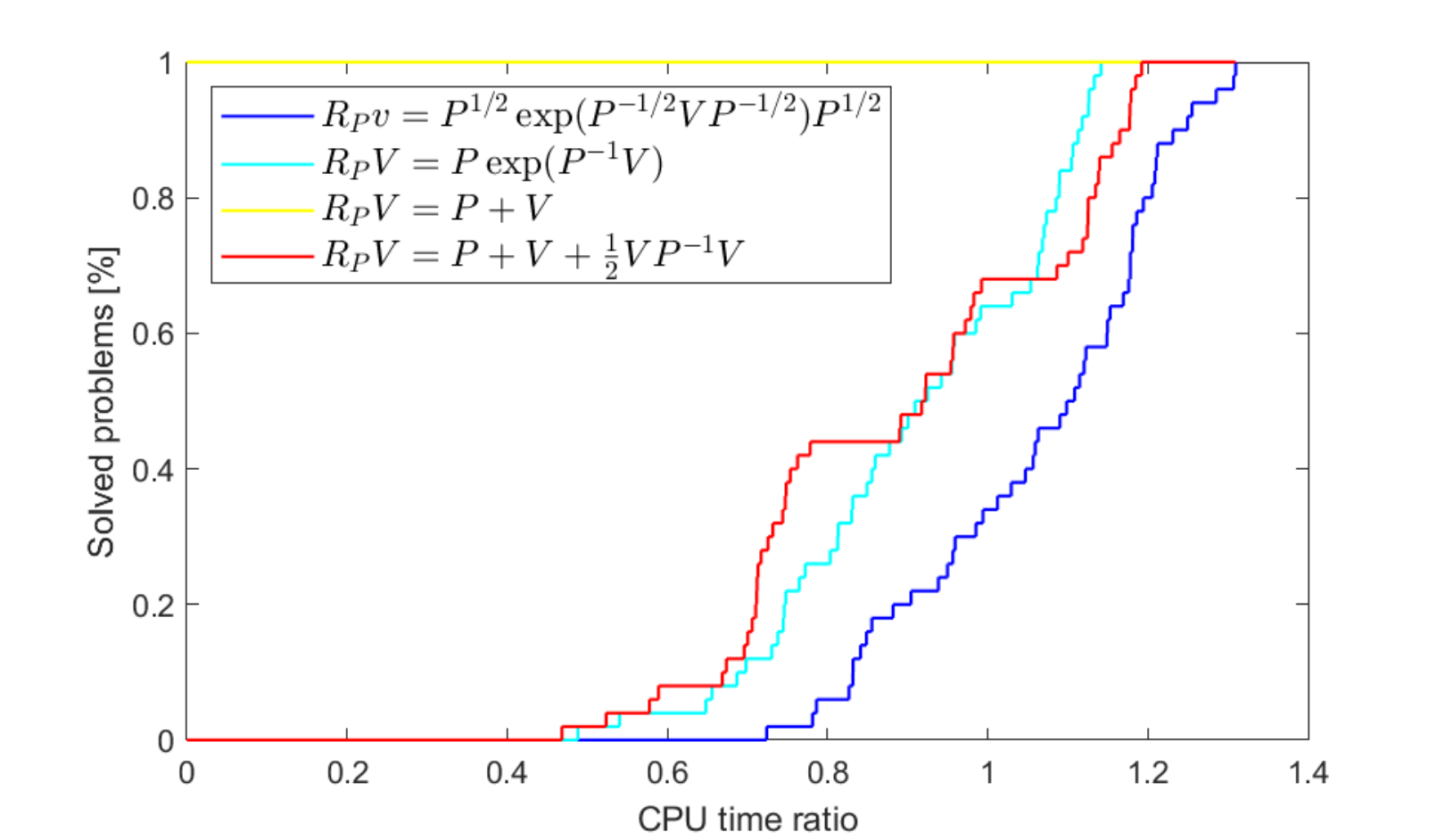}
	\caption{Performance profiles comparing retractions \eqref{eq:retexpcone}-\eqref{eq:ret2ordercone} for $f_2$  in the Algorithm \ref{AL:jogo_with_retraction}. }  \label{fig:conef2}
\end{figure}

 %%%%%%%%%%%%%%%%%%%%%%%%
\section{ Conclusions} \label{sec:conclusions}
%%%%%%%%%%%%%%%%%%%%%%%%%%
In this work, in order to find a singularity of a vector field defined on Riemannian manifolds, we presented a globalization of Newton method and established its global convergence with superlinear rate. The Algorithm \ref{AL:jogo_with_retraction} was designed with a general retraction in order to improve the performance of the analogous presented in \cite{Bortoloti2020}. As it can be seen in the Section \ref{Sec:Numerical_Experiment}, the numerical performance of Algorithm \ref{AL:jogo_with_retraction} is better than the one presented in \cite{Bortoloti2020}. We point out that the convergence analysis presented in \cite{Bortoloti2020} requeres nonsingularity of covariant derivative at cluster point. On the other hand, the condition \eqref{eq:NewtonDirectionCondction} in the Algorithm \ref{AL:dampedRetraction2} ensures that this hypothesis on the covariant derivative is not necessary to establish convergence, as it can be seen in the Theorem \ref{theoremAuxiliary123}. In addition, it is worth mentioning that condition  \eqref{eq:NewtonDirectionCondction} seems to avoid  Newton directions that can generate quite small step lengths. As a consequence,  Algorithm \ref{AL:dampedRetraction2} is more robust than Algorithm \ref{AL:jogo_with_retraction} in  number of solved problems as we can see in the experiments presented in Sections~\ref{sec:robust} and \ref{sec:rubustcone}.

\bibliographystyle{abbrv}
\bibliography{rnewton}

\end{document}